\documentclass{amsart}

\usepackage[margin=1in]{geometry}
\usepackage{amssymb}
\usepackage{hyperref}
\usepackage{amsmath,amscd}
\usepackage[enableskew]{youngtab}
\usepackage[all,cmtip]{xy}
\usepackage{shuffle}
\usepackage{color}
\usepackage{ytableau}

\newtheorem{theorem}{Theorem}[section]
\newtheorem{lemma}[theorem]{Lemma}
\newtheorem{proposition}[theorem]{Proposition}
\newtheorem{corollary}[theorem]{Corollary}

\newtheorem*{Young's Rule}{Young's Rule}

\theoremstyle{definition}

\newtheorem{example}[theorem]{Example}

\theoremstyle{remark}
\newtheorem{remark}[theorem]{Remark}

\numberwithin{equation}{section}

\def\ZZ{{\mathbb Z}} \def\CC{{\mathbb C}} \def\FF{{\mathbb F}}
\def\SS{{\mathfrak S}} \def\HH{\mathcal{H}_\bullet(0)} \def\H{\mathcal{H}}
\def\B{{\mathfrak S}^B} \def\HB{\mathcal{H}^B} \def\D{{\mathfrak S}^D} \def\HD{\mathcal{H}^D}
\def\P{\mathcal{P}} \def\C{\mathcal{C}} \def\M{\mathcal{M}} \def\IJS{\mathcal{P}_{I,J}^S}
  
\def\bP{\mathbf{P}} \def\bC{\mathbf{C}} \def\bM{\mathbf{M}}
\def\bPB{\mathbf{P}^B} \def\bCB{\mathbf{C}^B} \def\bMB{\mathbf{M}^B}
\def\bPD{\mathbf{P}^D} \def\bCD{\mathbf{C}^D} \def\bMD{\mathbf{M}^D}
\def\bX{\mathbf{X}} \def\bx{{\mathbf x}}
\def\FB{F^B} \def\MB{M^B} \def\FD{F^D} \def\MD{M^D} 
\def\bs{{\mathbf s}} \def\bh{{\mathbf h}} 
\def\bsB{\mathbf{s}^B} \def\bhB{\mathbf{h}^B} \def\bsD{\mathbf{s}^D} \def\bhD{\mathbf{h}^D}
 \def\inv{{\rm inv}}  
\def\neg{\mathrm{neg}}  \def\nsp{\mathrm{nsp}} 
\def\rad{\mathrm{rad}} \def\Hom{\operatorname{Hom}}
\def\cleq{\operatorname{\preccurlyeq}} 
\def\pib{\overline{\pi}} \def\rev{\operatorname{rev}}
\def\qand{\quad\text{and}\quad}
\DeclareMathOperator\modelsB{\models^B} \DeclareMathOperator\modelsD{\models^D}
\def\Sym{\mathrm{Sym}} \def\QSym{\mathrm{QSym}} \def\NSym{\mathbf{NSym}}

\newcommand\qbin[3]{\left[\begin{matrix} #1 \\ #2 \end{matrix} \right]_{#3}}

\begin{document}

\title[Tableau approach to representations of 0-Hecke algebras]{A tableau approach to the representation theory of 0-Hecke algebras}
\author{Jia Huang}
\address{Department of Mathematics and Statistics, University of Nebraska at Kearney, Kearney, NE 68849, USA}
\email{huangj2@unk.edu}
\thanks{The author thanks Victor Reiner for helpful conversations and encouragement.}

\begin{abstract}
A 0-Hecke algebra is a deformation of the group algebra of a Coxeter group. Based on work of Norton and Krob--Thibon, we introduce a tableau approach to the representation theory of 0-Hecke algebras of type A, which resembles the classic approach to the representation theory of symmetric groups by Young tableaux and tabloids. We extend this approach to type B and D, and obtain a correspondence between the representation theory of 0-Hecke algebras of type B and D and quasisymmetric functions and noncommutative symmetric functions of type B and D. Other applications are also provided. 
\end{abstract}

\maketitle

\section{Introduction}

The symmetric groups are fascinating and useful in many areas of mathematics. In particular, the representation theory of symmetric groups is related to the theory of symmetric functions via the classic Frobenius correspondence. More precisely, there is a graded Hopf algebra isomorphism between the Grothendieck group $G_0(\CC\SS_\bullet)$ associated with the tower $\CC\SS_0\hookrightarrow\CC\SS_1\hookrightarrow\CC\SS_2\hookrightarrow\cdots$ of (complex) group algebras of symmetric groups and the ring $\Sym$ of symmetric functions. This correspondence has many important applications in algebraic combinatorics.

The \emph{0-Hecke algebra $H_n(0)$ of type A} is generated by the \emph{bubble-sorting operators}, giving a deformation of the group algebra of $\SS_n$. In fact, there is a 0-Hecke algebra associated with any finite Coxeter system, whose representation theory was obtained by Norton~\cite{Norton} from a purely algebraic perspective. In type A this was also studied by Krob--Thibon~\cite{KrobThibon} and others from a more combinatorial point of view.

It is now well known to the experts that, analogously to the Frobenius correspondence for symmetric groups, the representation theory of 0-Hecke algebras of type A is related to the theory of \emph{quasisymmetric functions} and \emph{noncommutative symmetric functions}. In fact, there are two Grothendieck groups $G_0(\HH)$ and $K_0(\HH)$ associated with the tower $\HH: \H_0(0)\hookrightarrow \H_1(0)\hookrightarrow \H_2(0)\hookrightarrow \cdots$ of 0-Hecke algebras of type A. The induction and restriction of representations along the embedding $\H_m(0)\otimes\H_n(0)\hookrightarrow \H_{m+n}(0)$ endows $G_0(\HH)$ and $K_0(\HH)$ dual graded Hopf algebra structures which are isomorphic to the graded Hopf algebra $\QSym$ of \emph{quasisymmetric functions} and the dual graded Hopf algebra $\NSym$ of \emph{noncommutative symmetric functions}. This connects the representation theory of 0-Hecke algebras of type A to many combinatorial objects~\cite{H0CF, H0SR, H0QSchur}.

The proof of this correspondence, however, is somewhat scattered in the literature. Malvenuto and Reutenauer~\cite{MR} showed that $\QSym$ and $\NSym$ are dual graded Hopf algebras. Krob and Thibon~\cite{KrobThibon} showed that $G_0(\HH)$ and $K_0(\HH)$ are isomorphic to $\NSym$ and $\QSym$ as graded algebras via the \emph{quasisymmetric characteristic} $\mathrm{Ch}$ and \emph{noncommutative characteristic} $\mathbf{ch}$. Bergeron and Li~\cite{BergeronLi} introduced a set of conditions for a tower of graded algebras to have two Grothendieck groups being dual Hopf algebras, and verified these conditions for the tower of $0$-Hecke algebras of type A. Combining these results together one sees that $G_0(\HH)$ and $K_0(\HH)$ are isomorphic to $\NSym$ and $\QSym$ as dual graded Hopf algebras.

It is natural to ask how to extend this correspondence from type A to type B and D. We address this question by adopting a tableau approach to the representation theory of 0-Hecke algebras. It is well known that the Specht modules over symmetric groups can be constructed using Young tableaux and tabloids. In our earlier work~\cite{H0CF, H0SR}, we briefly mentioned an analogous tableau approach to the representation theory of 0-Hecke algebras of type A, based on work of Krob and Thibon~\cite{KrobThibon}. In this paper we provide a detailed exposition for this tableau approach and use it to obtain analogues of permutation modules, Specht modules, and Young's rule. We also use it to establish explicit coproduct formulas for $K_0(\HH)$ and $\NSym$, directly showing that the noncommutative characteristic $\mathbf{ch}: K_0(\HH)\to\NSym$ is an isomorphism of coalgebras. Note that the duality between this coalgebra isomorphism and the algebra isomorphism $\mathrm{Ch}: G_0(\HH)\to\QSym$ is not governed by Frobenius reciprocity, but rather by a non-standard property verified for 0-Hecke algebras of type A by Bergeron and Li~\cite{BergeronLi}. 

Analogously in type B and D, we use tableaux to study representations of 0-Hecke algebras and obtain analogues of permutation modules, Specht modules, and Young's rule. This naturally leads to a graded right $K_0(\H_\bullet(0))$-module structure on the Grothendieck group $K_0(\HB_\bullet(0))$ of the tower $\HB_\bullet(0)$ of 0-Hecke algebras of type B, as well as a type B analogue $\NSym^B$ of $\NSym$, which was previously introduced by Chow~\cite{Chow} as a free graded right $\NSym$-module but is now realized as certain formal power series in noncommutative variables. We define a type B noncommutative characteristic $\mathbf{ch}^B: K_0(\HB_\bullet(0))\to\NSym^B$ which gives a module isomorphism. Using  a type B analogue $\QSym^B$ of $\QSym$ introduced by Chow~\cite{Chow}, we define a type B quasisymmetric characteristic $\mathrm{Ch}^B: G_0(\HB_\bullet(0))\to\QSym^B$ which gives a comodule isomorphism dual to $\mathbf{ch}^B$ by the Frobenius reciprocity. We also obtain similar results in type D.

We next discuss some remaining problems in type B; similar problems also exist in type D. The $K_0(\H_\bullet(0))$-module structure on $K_0(\HB_\bullet(0))$ and the $G_0(\H_\bullet(0))$-comodule structure on $G_0(\HB_\bullet(0))$ are defined by induction and restriction of representations of 0-Hecke algebras along embeddings of the form $\HB_m(0)\otimes\H_n(0)\hookrightarrow\HB_{m+n}(0)$, where $\HB_m(0)\otimes\H_n(0)$ is identified with a parabolic subalgebra of $\HB_{m+n}(0)$. One can similarly define a $G_0(\H_\bullet(0))$-module structure on $G_0(\HB_\bullet(0))$ and a $K_0(\H_\bullet(0))$-comodule structure on $K_0(\HB_\bullet(0))$. However, it is not clear whether $\mathrm{Ch}^B$ is a module isomorphism or $\mathbf{ch}^B$ is a comodule isomorphism, nor is the duality between them as the Frobenius reciprocity does not apply to such situation. We will solve these problems using a different approach in another paper~\cite{ChH0}. But we do not have an algebra or coalgebra structure on the Grothendieck groups of $\HB_\bullet(0)$, since we are not able to find embeddings of the form $\HB_m(0)\otimes\HB_n(0)\hookrightarrow\HB_{m+n}(0)$. 

Some other results follow naturally from our tableau approach. First, it is known that the antipode of the Hopf algebra $\Sym$ corresponds to tensoring representations of a symmetric group with the sign representation. We show that the antipodes of the Hopf algebras $\QSym$ and $\NSym$ can be interpreted by taking transpose of the tableaux that we use to study representations of 0-Hecke algebras of type A. We also find similar symmetries among tableaux in type B and D, respectively. We describe these symmetries uniformly using two automorphisms of 0-Hecke algebras studied by Fayers~\cite{Fayers}.

Next, the Specht modules of the symmetric group $\SS_n$ can also be obtained by using the $\SS_n$-action on the polynomial ring $\FF[x_1,\ldots,x_n]$. We provide an analogue of this using the $\H_n(0)$-action on $\FF[x_1,\ldots,x_n]$ via the \emph{Demazure operators}.

Next, Lam, Lauve, and Sottile~\cite{Skew} studied skew elements in the dual graded Hopf algebras $\QSym$ and $\NSym$, and provided a formula for the former. Now we provide a formula for the latter thanks to our coproduct formula for $\NSym$.

Finally, applying the quasisymmetric characteristic map to certain representations of 0-Hecke algebras gives some combinatorial identities involving quasisymmetric functions, $q$-multinomial coefficients, and $q$-ribbon numbers.

This paper is structured in the following way. After providing some preliminaries in Section~\ref{sec:pre}, we use tableaux to study the representation theory of 0-Hecke algebras of type A, B, and D in Section~\ref{sec:rep}, and investigate its connections with quasisymmetric functions and noncommutative symmetric functions of type A, B, and D in Section~\ref{sec:Char}. Other applications are included in Section~\ref{sec:other}. 

\section{Preliminaries}\label{sec:pre}

\subsection{Algebras and their representations}\label{sec:algebra}

We review some general results on the representation theory of (unital associative) algebras. See, e.g., \cite[\S I]{ASS} for more details. 
Let $A$ be an algebra over a field $\FF$ and let $M$ be a (left) $A$-module. 
We say $M$ is \emph{simple} if it has no submodule other than $0$ and itself. 
We say $M$ is \emph{semisimple} if it is a direct sum of simple $A$-modules. 
The algebra $A$ is \emph{semisimple} if itself is a semisimple $A$-module, or equivalently, if every $A$-module is semisimple. 
The \emph{radical} $\rad(M)$ of $M$ is the intersection of all maximal submodules of $M$. The \emph{top of $M$} is ${\rm top}(M):=M/\rad(M)$, which is always semisimple. 
We say $M$ is \emph{indecomposable} if a decomposition $M=L\oplus N$ of $M$ into a direct sum of two submodules $L,N$ implies either $L=0$ or $N=0$. 
We say $M$ is \emph{projective} if $M$ is a direct summand of a free $A$-module.

Let $B$ be a subalgebra of $A$. For any $A$-module $M$ and $B$-module $N$, the induced $A$-module $N\uparrow\,_B^A$ is defined as $A\otimes_B N$, and the restricted $B$-module $M\downarrow\,_B^A$ is $M$ itself viewed as a $B$-module. 
When $A$ and $B$ are group algebras the Frobenius Reciprocity asserts ${\rm Hom}_A(N\uparrow\,_B^A,M) \cong {\rm Hom}_B(N,M\downarrow\,_B^A)$, which also holds for 0-Hecke algebras~\cite[Example~4.3]{BergeronLi}, \cite[Remark~2.3.7]{ChH0}. 
But in general ${\rm Hom}_A(M, N\uparrow\,_B^A) \not\cong {\rm Hom}_B(M\downarrow\,_B^A,N)$.

Let $A=\P_1\oplus\cdots\oplus \P_m$ be a decomposition of $A$ into a direct sum of indecomposable submodules $\P_1,\ldots,\P_m$. In general $\P_i$ is \emph{not} simple, but its top $\C_i:={\rm top}(\P_i)$ is, for all $i=1,\ldots,m$. Moreover, every simple $A$-module is isomorphic to some $\C_i$, and every projective indecomposable $A$-module is isomorphic to some $\P_j$. 

The \emph{Grothendieck group $G_0(A)$ of the category of all finitely generated $A$-modules} is defined as the abelian group $F/F'$, where $F$ is the free abelian group on the isomorphism classes $[M]$ of all finitely generated $A$-modules $M$, and $F'$ is the subgroup of $F$ generated by the elements $[M]-[L]-[N]$ corresponding to all short exact sequences $0\to L\to M\to N\to0$ of finitely generated $A$-modules. We write $M$ for the image of $[M]$ in $F/F'$ if it causes no confusion. The \emph{Grothendieck group $K_0(A)$ of the category of all finitely generated projective $A$-modules} is defined similarly. It turns out that $G_0(A)$ and $K_0(A)$ are free abelian groups on $\{C_1,\ldots,\C_m\}$ and $\{P_1,\ldots,\P_m\}$, respectively. Since short exact sequences of projective modules split, one has $P=P'+P''$ in $K_0(A)$ if and only if $P=P'\oplus P''$. 

Associated with a tower of algebras $A_\bullet: A_0\hookrightarrow A_1\hookrightarrow A_2\cdots$ are two Grothendieck groups
\[ G_0(A_\bullet) := \bigoplus_{n\geq0} G_0(A_n) \qand
K_0(A_\bullet) := \bigoplus_{n\geq0} K_0(A_n). \]
Suppose that there is an algebra embedding $A_m\otimes A_n\hookrightarrow A_{m+n}$ for all pairs of nonnegative integers $m$ and $n$. Bergeron and Li~\cite{BergeronLi} showed that, under certain assumptions---which are, in particular, satisfied by the tower of group algebras of symmetric groups and the tower of 0-Hecke algebras of type A---the Grothendieck groups $G_0(A_\bullet)$ and $K_0(A_\bullet)$ become dual graded Hopf algebra whose product and coproduct are given by 
\[ M\,\widehat\otimes\, N := (M\otimes N)\uparrow\,_{A_m\otimes A_n}^{A_{m+n}} \qand 
\Delta M := \sum_{0\le i\le m} M\downarrow\,_{A_i\otimes A_{m-i}}^{A_m} \]
for all $M\in G_0(A_m)$ (or $K_0(A_m)$) and $N\in G_0(A_n)$ (or $K_0(A_n)$). The duality is given by the pairing
\[ \langle P,M \rangle:=
\begin{cases}
\dim_\FF\Hom_{A_m} (P,M), & {\rm if}\ P\in K_0(A_m), M\in G_0(A_m),\\
0, & {\rm otherwise}.
\end{cases} \]

\subsection{Coxeter groups}

A \emph{finite Coxeter group} $W$ is a finite group generated by a set $S$ with relations 
$s^2=1$ for all $s\in S$ and $(sts\cdots)_{m_{st}} = (tst\cdots)_{m_{st}}$ for all distinct $s,t\in S$, where $m_{st} = m_{ts} \in\{2,3,\ldots\}$ and $(aba\cdots)_m$ denotes an alternating product of $m$ terms. The pair $(W,S)$ is called a \emph{finite Coxeter system}. An element $w\in W$ can be expressed as $w=s_{i_1}\cdots s_{i_k}$ where $s_{i_1},\ldots,s_{i_k}$ belong to $S$. Such an expression of $w$ is called \emph{reduced} if $k$ is as small as possible, and the smallest $k$ is the \emph{length} $\ell(w)$ of $w$.

If $s\in S$ then $\ell(ws)=\ell(w)-1$ or $\ell(ws)=\ell(w)+1$. The former happens if and only if $w$ has a reduced expression ended with $s$, and in such case $s$ is called a \emph{descent} of $w$. The \emph{descent set} $D(w)$ of $w$ consists of all its descents.

Let $I\subseteq S$ and $I^c:=S\setminus I$. The \emph{parabolic subgroup} $W_I$ is the subgroup of $W$ generated by $I$, and $(W_I,I)$ is a Coxeter subsystem of $(W,S)$. 
An element $w\in W$ can be written uniquely as $w = z u$ where $z\in W^I$ and $u\in W_I$, and one has $\ell(w) = \ell(z) + \ell(u)$. Here $W^I:=\{ z\in W: D(z)\subseteq I^c \}$ is the set of length-minimal representatives for left $W_I$-cosets.

\begin{theorem}[{Bj\"orner and Wachs~\cite[Theorem~6.2]{BjornerWachs}}]\label{thm:interval}
Let $I\subseteq J\subseteq S$. Then $\{ w\in W: I\subseteq D(w)\subseteq J\}$ is an interval $[w_0(I),w_1(J)]$ under the left weak order of $W$, where $w_0(I)$ and $w_1(J)$ are the longest elements in $W_{I}$ and $W^{J^c}$, respectively.
\end{theorem}

In particular, the \emph{descent class} $\{w\in W: D(w)=I\}$ of $I\subseteq S$ is an interval $[w_0(I),w_1(I)]$ under the left weak order of $W$. Let $w_0:=w_0(S)$ be the longest element of $W$. The automorphism $w\mapsto w_0ww_0$ of $W$ is determined by an automorphism $\sigma$ of the Coxeter diagram of $(W,S)$---see, for example, Fayers~\cite[Proposition~2.4]{Fayers}. It is well known that $w\mapsto ww_0$ and $w\mapsto w_0w$ give two anti-isomorphisms of the weak order of $W$, sending the descent class of $I$ to the descent classes of $I^c$ and $\sigma(I^c)$, respectively.

The \emph{symmetric group $\SS_n$} consists of all permutations of the set $[n]:=\{1,2,\ldots,n\}$. It is the Coxeter group of type $A_{n-1}$, generated by the \emph{adjacent transpositions} $s_i:=(i,i+1)$, $i=1,2,\ldots,n-1$, with the quadratic relations $s_i^2=1$, $1\le i \le n-1$, and the braid relations 
\[\begin{cases}
s_is_{i+1}s_i = s_{i+1}s_is_{i+1},\ 1\le i\le n-2,\\
s_is_j=s_js_i,\ 1\le i,j\le n-1,\ |i-j|>1.
\end{cases}\]
The length of $w\in\SS_n$ is $\inv(w):=\#\{(i,j):1\le i<j\le n,\ w(i)>w(j)\}$ and the descent set of $w\in\SS_n$ is $D(w)=\{ 1\le i\le n-1: w(i)>w(i+1)\}$ where we identify a generator $s_i$ with its subscript $i$.

It is often convenient to use compositions to study symmetric groups. A \emph{composition} is a sequence $\alpha=(\alpha_1,\ldots,\alpha_\ell)$ of positive integers. The \emph{length} of $\alpha$ is $\ell(\alpha):=\ell$ and the \emph{size} of $\alpha$ is $|\alpha|:=\alpha_1+\cdots+\alpha_\ell$. We say $\alpha$ is a composition of $n$ and write $\alpha\models n$ if $|\alpha|=n$. The \emph{descent set} of $\alpha$ is 
$
D(\alpha):=\{\alpha_1,\alpha_1+\alpha_2,\ldots,\alpha_1+\cdots+\alpha_{\ell-1}\}.
$
One sees that $\alpha\mapsto D(\alpha)$ is a bijection between compositions of $n$ and subsets of $[n-1]:=\{1,\ldots,n-1\}$. We write $\alpha\cleq\beta$ if $\alpha$ and $\beta$ are two compositions of the same size such that $D(\alpha)\subseteq D(\beta)$, or in other words, $\alpha$ is refined by $\beta$.

Given a composition $\alpha=(\alpha_1,\ldots,\alpha_\ell)$ of $n$, there is a parabolic subgroup $\SS_\alpha\cong\SS_{\alpha_1}\times\cdots\times\SS_{\alpha_\ell}$ of $\SS_n$ generated by $\{s_i:i\in [n-1]\setminus D(\alpha)\}$. A complete set of length-minimal representatives for left $\SS_\alpha$-cosets in $\SS_n$ is given by
$ \SS^\alpha:= \left \{ w \in\SS_n: D(w)\subseteq D(\alpha) \right \}. $
The \emph{descent class of $\alpha$} consists of all permutations $w$ in $\SS_n$ with $D(w)=D(\alpha)$, which is an interval under the left weak order of $\SS_n$, denoted by $[w_0(\alpha),w_1(\alpha)]$. 

\subsection{Tableaux and Symmetric functions}
A sequence $\lambda=(\lambda_1,\ldots,\lambda_\ell)$ of positive integers such that $\lambda_1\ge\cdots\ge\lambda_\ell$ and  $|\lambda|:=\lambda_1+\cdots+\lambda_\ell=n$ is called a \emph{partition of $n$} and denoted by $\lambda\vdash n$. A partition $\lambda$ is identified with its \emph{Young diagram}, and its \emph{transpose} (or \emph{conjugate}) $\lambda^t$ is obtained by reflecting its Young diagram along the main diagonal. Deleting a Young diagram $\mu$ contained in $\lambda$ gives a \emph{skew diagram} $\lambda/\mu$. See examples below.\vskip5pt
\[ \begin{matrix}
\raisebox{-15pt}{\young(\hfill\hfill\hfill\hfill,\hfill\hfill\hfill,\hfill\hfill,\hfill\hfill)} \qquad&
\raisebox{-15pt}{\young(\hfill\hfill\hfill\hfill,\hfill\hfill\hfill\hfill,\hfill\hfill,\hfill)} \qquad&
\raisebox{-15pt}{\young(::\hfill\hfill,:\hfill\hfill,\hfill\hfill,\hfill\hfill)} \\ \\
\lambda = (4,3,2,2) \qquad& \lambda^t = (4,4,2,1) \qquad & (4,3,2,2)/(2,1)
\end{matrix}\]\vskip5pt

A \emph{semistandard tableau $\tau$ of shape $\lambda/\mu$} is a filling of the skew diagram $\lambda/\mu$ by positive integers such that every row weakly increases from left to right and every column strictly increases from top to bottom. Reading these integers from the bottom row to the top row and proceeding from left to right within each row gives the \emph{reading word} $w(\tau)$ of $\tau$. A semistandard tableau $\tau$ of shape $\lambda/\mu$ is a \emph{standard} if its reading word $w(\tau)$ is a permutation in $\SS_n$ where $n=|\lambda|-|\mu|$. The following example shows a semistandard tableau  and a standard tableau, both of skew shape $(4,3,2,2)/(2,1)$.
\vskip5pt\[ \young(::46,:15,22,34) \qquad\qquad\qquad \young(::68,:17,23,45) \]\vskip5pt

Let $X=\{x_1,x_2,\ldots\}$ be a totally ordered set of commutative variables. The \emph{Schur function} of skew shape $\lambda/\mu$ is the sum of $x_\tau$ for all semistandard tableaux $\tau$ of shape $\lambda/\mu$, where $x_{\tau}:=x_{w_1}\cdots x_{w_n}$ if $w(\tau)=w_1\cdots w_n$.
The \emph{ring of symmetric functions}, denoted by $\Sym$, is the free $\ZZ$-module with a basis consisting of Schur functions $s_\lambda:=s_{\lambda/\emptyset}$ for all partitions $\lambda$. The multiplication of formal power series defines a product for $\Sym$. The coproduct $\Delta$ of $\Sym$ is defined by $\Delta f(X):=f(X+Y)$ for all $f\in\QSym$, where $X+Y=\{x_1,x_2,\ldots,y_1,y_2,\ldots\}$ is a totally ordered set of commutative variables. It is well known that $\Sym$ is a self-dual graded Hopf algebra containing all skew Schur functions. 

\subsection{Representation theory of symmetric groups}
We review from Sagan~\cite{Sagan} and Stanley~\cite[Chapter 7]{EC2} the representation theory of symmetric groups and its connections with symmetric functions.

The Grothendieck groups $G_0(\CC\SS_\bullet)$ and $K_0(\CC\SS_\bullet)$ of the tower of algebras $\CC\SS_\bullet: \CC\SS_0\hookrightarrow\CC\SS_1\hookrightarrow\CC\SS_2\hookrightarrow\cdots$ are the same, since the group algebra $\CC\SS_n$ is semisimple. The simple $\CC\SS_n$-modules $S^\lambda$ are indexed by partitions $\lambda$ of $n$ and can be constructed using Young tableaux and tabloids of shape $\lambda$, as described below.

Let $\lambda\vdash n$. A \emph{Young tableau of shape $\lambda$} is a filling of the Young diagram of $\lambda$ with $1,2,\ldots,n$, each number occurring exactly once. 
A \emph{tabloid of shape $\lambda$} is the equivalence classes of Young tableaux of shape $\lambda$ under permutations of entries in the same row. The \emph{permutation module} $M^\lambda$ has a basis consisting of tabloids of shape $\lambda$, with an $\SS_n$-action by permuting the entries in these tabloids. The \emph{polytabloid} $e_t$ of a Young tableau $t$ of shape $\lambda$ is 
\[ e_t := \sum_{ w \in C_t } {\rm sgn}(w) w(t) \quad \text{inside}\ \CC\SS_n\]
where $C_t$ is the subgroup of $\SS_n$ preserving every column of $t$. The \emph{Specht module} $S^\lambda$ is the submodule of $M^\lambda$ spanned by polytabloids $e_t$ for all Young tableaux $t$ of shape $\lambda$ and has a basis consisting of those $e_t$ with $t$ standard. 

\begin{Young's Rule}
If $\lambda$ is a partition of $n$ then
$ M^\lambda \cong 1 \uparrow\,_{\SS_\lambda}^{\SS_n} \cong \bigoplus_{\lambda\unlhd\mu} K_{\mu\lambda} S^\mu $
where $1$ denotes the trivial representation, $\unlhd$ is the dominance order, and $K_{\lambda\mu}$ is the \emph{Kostka number}. 
\end{Young's Rule}

Using the embeddings $\SS_m\times\SS_n\cong \SS_{m,n}\subseteq \SS_{m+n}$ one defines a product and a coproduct for $G_0(\CC\SS_\bullet)$ by 
\[ M\,\widehat\otimes\, N := (M\otimes N)\uparrow\,_{\SS_m\times\SS_n}^{\SS_{m+n}} \qand
\Delta M := \sum_{0\le i\le m} M\downarrow\,_{\SS_i\times\SS_{m-i}}^{\SS_m} \]
for all $M\in G_0(\CC\SS_m)$ and $N\in G_0(\CC\SS_n)$. 
The \emph{Frobenius characteristic} ch: $G_0(\CC\SS_\bullet)\to\Sym$ sends the Specht module $S^\lambda$ to the Schur function $s_\lambda$ for all partitions $\lambda$, giving an isomorphism of graded Hopf algebras.

\section{Representation theory of 0-Hecke algebras}\label{sec:rep}

In this section we give a tableau approach to the representation theory of 0-Hecke algebras of type A, B, and D.

\subsection{0-Hecke algebras of finite Coxeter systems}\label{sec:HW0}
Let $W$ be a finite Coxeter group generated by a set $S$ with relations $s^2=1$ for all $s\in S$ and $(sts\cdots)_{m_{st}} = (tst\cdots)_{m_{st}}$ for all distinct $s,t\in S$.

The \emph{0-Hecke algebra} $\H_W(0)$ of the Coxeter system $(W,S)$ is a deformation of the group algebra of $W$, defined as an algebra over a field $\FF$ generated by $\{\pi_s:s\in S\}$ with relations $\pi_s^2 = \pi_s$ for all $s\in S$ and $(\pi_s\pi_t\pi_s\cdots)_{m_{st}} = (\pi_t\pi_s\pi_t\cdots)_{m_{st}}$ for all distinct $s, t \in S$. 

Let $\pib_s := \pi_s-1$. Then $\pib_s\pi_s=0$.
One can check that $\H_W(0)$ is also generated by $\{\pib_s:s\in S\}$ with relations $\pib_s^2 = -\pib_s$ for all $s\in S$ and $(\pib_s\pib_t\pib_s\cdots)_{m_{st}} = (\pib_t\pib_s\pib_t\cdots)_{m_{st}}$ for all distinct $s, t \in S$.

If an element $w\in W$ has a reduced expression $w=s_{i_1}\cdots s_{i_k}$ then $\pi_w:=\pi_{s_1}\cdots\pi_{s_k}$ and $\pib_w:=\pib_{s_1}\cdots\pib_{s_k}$ are both well defined, thanks to the Word Property of $W$~\cite[Theorem~3.3.1]{BjornerBrenti}. Moreover, the two sets $\{\pi_w:w\in W\}$ and $\{\pib_w:w\in W\}$ are both $\FF$-bases for the 0-Hecke algebra $\H_W(0)$.

The representation theory of $\H_W(0)$ was studied by Norton~\cite{Norton}. Below is a summary of her main results.

\begin{theorem}\label{thm:Norton}
(i) Let $I\subseteq S$. The $\H_W(0)$-module $\P_I^S:=\H_W(0)\pib_{w_0(I)}\pi_{w_0(I^c)}$ is indecomposable and has a basis
\[
\{\pib_w\pi_{w_0(I^c)}: w\in W,\ D(w) = I \}.
\] 
 The top $\C_I^S$ of $\P_I^S$ is a one-dimensional $\H_W(0)$-module on which $\pib_i$ acts by either $-1$ if $i\in I$ or $0$ otherwise.

\noindent (ii) Let $J\subseteq S$. The $\H_W(0)$-module $\M_J^S:=\H_W(0)\pi_{w_0(J^c)}$ has a basis 
\[ \{ \pib_w\pi_{w_0(J^c)}: w\in W,\ D(w)\subseteq J \} \]
and decomposes as
\[ \M_J^S = \bigoplus_{I\subseteq J} \P_I^S. \]
In particular, this gives a decomposition of $\H_W(0)$ itself as an $\H_W(0)$-module into a direct sum of indecomposable submodules when $J=S$.
\end{theorem}

It follows from Theorem~\ref{thm:Norton} that $\{\P_I^S: I\subseteq S\}$ is a complete list of non-isomorphic projective indecomposable $\H_W(0)$-modules and $\{\C_I^S: I\subseteq S\}$ is a complete list of non-isomorphic simple $\H_W(0)$-modules. 
 
We next define a family of $\H_W(0)$-modules which simultaneously generalize the $\H_W(0)$-modules $\P_I^S$ and $\M_J^S$. Let $I$ and $J$ be two subsets of $S$. We define $\IJS:=\H_W(0) \pib_{w_0(I)} \pi_{w_0(J\setminus I)}$. One sees that $\P_I^S = \P_{I,S}^S$ and $\M_J^S = \P_{\emptyset,J^c}^S$. We may assume $I\subseteq J\subseteq S$, without loss of generality, since  $\IJS = \P_{I,I\cup J}^S$. We provide the following result on $\IJS$, which generalizes Theorem~\ref{thm:Norton} (ii) and will be used later.

\begin{theorem}\label{thm:IJS}
(i) If $I, J \subseteq S$ then $\IJS$ has a basis 
\begin{equation}\label{eq:BasisIJS}
\{ \pib_w\pi_{w_0(J\setminus I)}: w\in W,\ I\subseteq D(w) \subseteq J^c \cup I\}
\end{equation}

\noindent(ii) If $I\subseteq J \subseteq S$ then 
\[
\P_I^J \uparrow\,_{\H_{W_J(0)}}^{\H_W(0)} \cong \IJS\cong  \bigoplus_{K\subseteq J^c} \P_{I\cup K}.
\]
\end{theorem}

\begin{proof}
Let $I,J\subseteq S$. If $w\in W$ satisfies $I\subseteq D(w)\subseteq J^c\cup I$ then for all $s\in S$,
\begin{equation}\label{eq:HW0action}
\pib_s \pib_w\pi_{w_0(J\setminus I)} = \begin{cases}
-\pib_w\pi_{w_0(J\setminus I)}, & \ell(sw)<\ell(w), \\
0, &  \ell(sw)>\ell(w),\ D(sw) \not\subseteq J^c\cup I,\\
\pib_{sw}\pi_{w_0(J\setminus I)}, & \ell(sw)>\ell(w),\ D(sw)\subseteq J^c\cup I.
\end{cases} 
\end{equation}
This implies that \eqref{eq:BasisIJS} is a basis for $\IJS$ and proves (i).

Assume $I\subseteq J\subseteq S$ below. By the decomposition
\[
\H_W(0) = \bigoplus_{z\in W^J} \pib_z \H_{W_J(0)}
\]
$\P_I^J \uparrow\,_{\H_{W_J(0)}}^{\H_W(0)}$ is isomorphic to the submodule $\M$ of $\H_W(0)$ with a basis
\[
\{ \pib_z \pib_u \pi_{w_0(J\setminus I)}: z\in W^J,\ u\in W_J,\ D(u) = I \}.
\]
On the other hand, each $w\in W$ can be written uniquely as $w=zu$ where $z\in W^J$ and $u\in W_J$. For any $s\in J$ one has $s\in D(u) \Leftrightarrow s\in D(zu)$ since
\[
\ell(zu) = \ell(z) + \ell(u) \quad \text{and} \quad \ell(zus) = \ell(z) + \ell(us).
\]
It follows that $D(u)\subseteq D(w)\subseteq J^c\cup D(u)$. Thus the basis \eqref{eq:BasisIJS} for $\IJS$ is the same as the basis for $\M$, i.e. $\IJS=\M$. This proves the first desired isomorphism in (ii).

Next, we list the subsets of $J^c$ as $K_1,\ldots,K_\ell$ such that $K_i\subseteq K_j$ implies $i\le j$. Let $\M_i$ be the $\FF$-span of the disjoint union
\[
\bigsqcup_{j\ge i} \left\{ \pib_w \pib_{w_0(J\setminus I)}: w\in W,\ D(w) = I\cup K_j \right\}.
\]
Since $D(w)\subseteq D(sw)$ for any $s\in S$ and $w\in W$, it follows from \eqref{eq:HW0action} that there is a filtration
\[
\IJS =\M_1\supseteq \M_2\supseteq\cdots\supseteq \M_\ell\supseteq \M_{\ell+1}=0
\]
of $\H_W(0)$-modules. A basis for $\M_i/\M_{i+1}$ is given by $\{\pib_w \pib_{w_0(J\setminus I)}: w\in W,\ D(w) = I\cup K_i\}$ and one has $\M_i/\M_{i+1} \cong \P_{I\cup K_i}^S$ by \eqref{eq:HW0action}. This implies the second desired isomorphism in (ii).
\end{proof}

\subsection{0-Hecke algebras of type A}\label{sec:H0A}
Now we take the Coxeter system $(W,S)$ to be of type $A_{n-1}$, i.e., let $W=\SS_n$ and $S=\{s_1,\ldots,s_{n-1}\}$ where $s_i=(i,i+1)$. The $0$-Hecke algebra $\H_n(0):=\H_W(0)$ is generated by $\pi_1,\ldots,\pi_{n-1}$ with the quadratic relations $\pi_i^2=\pi_i$, $1\le i\le n-1$, and the same braid relations as $\SS_n$. A generator $\pi_i$ can be realized as the \emph{$i$th bubble-sorting operator} which swaps the adjacent positions $a_i$ and $a_{i+1}$ in a word $a_1\cdots a_n\in\ZZ^n$ if $a_i<a_{i+1}$, or fixes the word otherwise. Another generating set for $\H_n(0)$ consists of $\pib_i:=\pi_i-1$, $1\le i\le n-1$, with the quadratic relations $\pib_i^2=-\pib_i$ and the same braid relations as $\SS_n$. The two sets $\{\pi_w:w\in\SS_n\}$ and $\{\pib_w:w\in\SS_n\}$ are both bases for $\H_n(0)$.

The representation theory of $\H_n(0)$ is a special case of the result of Norton~\cite{Norton} and can be rephrased in a more combinatorial fashion. Based on work of Krob and Thibon~\cite{KrobThibon}, we briefly mentioned in~\cite{H0CF, H0SR} a tableau approach to the representation theory of $\H_n(0)$ but did not give details nor make any substantial use of it. Here we provide a more extensive discussion of this approach and will use it to establish other results in later sections.

We first review some notation. Let $\alpha=(\alpha_1,\ldots,\alpha_\ell) \models n$. The \emph{parabolic subalgebra} $\H_\alpha(0)\cong \H_{\alpha_1}(0)\otimes\cdots\otimes \H_{\alpha_\ell}(0)$ of $\H_n(0)$ is generated by $\{\pib_i:i\in [n-1]\setminus D(\alpha)\}$ and has bases $\{\pi_w:w\in\SS_\alpha\}$ and $\{\pib_w:w\in\SS_\alpha\}$.

A composition $\alpha=(\alpha_1,\ldots,\alpha_\ell)$ of $n$ can be identified with a \emph{ribbon diagram}, i.e., a connected skew diagram without $2\times2$ boxes, such that the rows have lengths $\alpha_1,\ldots,\alpha_\ell$, ordered from bottom to top. The \emph{reverse} of $\alpha$ is $\rev(\alpha):=(\alpha_\ell,\ldots,\alpha_1)$. The \emph{complement} of $\alpha$, denoted by $\alpha^c$, is the unique composition of $n$ whose descent set is $[n-1]\setminus D(\alpha)$. The \emph{transpose} or \emph{conjugate} of $\alpha$ is $\alpha^t := \rev(\alpha^c) = (\rev(\alpha))^c$, or equivalently, the transpose of the ribbon diagram of $\alpha$. An example is below.
\vskip5pt\[ \begin{matrix} 
\young(:::\hfill,:::\hfill,:\hfill\hfill\hfill,\hfill\hfill) & &
\young(::\hfill\hfill,\hfill\hfill\hfill,\hfill,\hfill) & &
\young(:::\hfill,::\hfill\hfill,::\hfill,\hfill\hfill\hfill) \\ 
\alpha=(2,3,1,1) & & \rev(\alpha)=(1,1,3,2) & & \alpha^t=(3,1,2,1)
\end{matrix}\]\vskip5pt
If $\alpha=(\alpha_1,\ldots,\alpha_\ell)$ and $\beta=(\beta_1,\ldots,\beta_k)$ are two ribbons, then one can glue them in two ways and obtain two ribbons
\[ \alpha\cdot\beta := (\alpha_1,\ldots,\alpha_\ell,\beta_1,\ldots,\beta_k) \qand 
\alpha\rhd\beta := (\alpha_1,\ldots,\alpha_{\ell-1},\alpha_\ell+\beta_1,\beta_2,\ldots,\beta_k). \]

By Norton's result, the projective indecomposable $\H_n(0)$-modules and simple $\H_n(0)$-modules are given by $\P_I^S$ and $\C_I^S$ for all $I\subseteq S=\{s_1,\ldots,s_{n-1}\}$. One can index these modules by compositions of $n$ thanks to the bijection $\alpha\mapsto D(\alpha)$. In Section~\ref{sec:HW0} we also studied $\H_n(0)$-modules $\IJS$ for all  $I,J\subseteq S$, generalizing the projective indecomposable $\H_n(0)$-modules. We give a tableau approach to these $\H_n(0)$-modules below.

Let $\alpha=\alpha^1 \oplus \cdots \oplus \alpha^k$ be a \emph{generalized ribbon} with connected components $\alpha^i\models n_i$ for $i=1,\ldots,k$, that is, a skew diagram whose connected components are ribbons $\alpha^1,\ldots,\alpha^k$, such that $\alpha^{i+1}$ is strictly to the northeast of $\alpha^i$ for $i=1,\ldots,k-1$. The \emph{size} of $\alpha$ is $|\alpha|:=n_1+\cdots+n_k$. We recursively define 
\[ [\alpha] := \left \{\beta\cdot\alpha^k,\ \beta\rhd\alpha^k: \beta\in [\alpha^1\oplus\cdots\oplus\alpha^{k-1}] \right\} \]
with $[\alpha^1]:=\{\alpha^1\}$. For example, the generalized ribbon $\alpha=2\oplus22\oplus32$ of size $|\alpha|=11$ is represented by the following diagram and has $[\alpha]=\{22232,4232,2252,452\}$.
\vskip5pt\[ \begin{matrix}
\young(\hfill\hfill) & \oplus & 
& \raisebox{-5pt}{\young(:\hfill\hfill,\hfill\hfill)} & \oplus &
\raisebox{-5pt}{\young(::\hfill\hfill,\hfill\hfill\hfill)} & = & 
\raisebox{-15pt}{\young(:::::::::\hfill\hfill,:::::::\hfill\hfill\hfill,::::\hfill\hfill,:::\hfill\hfill,\hfill\hfill) }
\end{matrix} \]\vskip5pt
One sees that standard tableaux of shape $\alpha$ are in bijection with permutations in the descent classes of $\beta$ in $\SS_n$ for all $\beta\in[\alpha]$. We define $\bP_\alpha$ as the vector space with a basis consisting of standard tableaux of shape $\alpha$. If $i\in[n-1]$ and $\tau$ is a standard tableau of shape $\alpha$ then we define
\begin{equation}\label{eq:RibbonAction}
\pib_i(\tau): = \begin{cases}
-\tau, &  \textrm{if $i$ is in a higher row of $\tau$ than $i+1$},\\
0, & \textrm{if $i$ is in the same row of $\tau$ as $i+1$}, \\
s_i(\tau), & \textrm{if $i$ is in a lower row of $\tau$ than $i+1$}.
\end{cases}
\end{equation}
One can directly check that this gives an $\H_n(0)$-action on $\bP_\alpha$, which is in fact a consequence of the following result.

\begin{theorem}\label{thm:IndP}
(i) If $\alpha\models n$ then $\bP_\alpha \cong \P_I^S$, where $I=\{s_i:i\in D(\alpha)\}$.

\noindent (ii) Let $\alpha=\alpha^1 \oplus \cdots \oplus \alpha^k$ be a generalized ribbon of size $n$ with connected components $\alpha^i\models n_i$ for $i=1,\ldots,k$. Let 
\[ I = \{s_i: i\in D(\alpha^1\rhd\cdots\rhd \alpha^k)\} \qand
J = \{s_j:j\in D(1^{n_1}\rhd\cdots\rhd 1^{n_k}) \}. \]
Then $\bP_\alpha$ is a well-defined $\H_n(0)$-module by \eqref{eq:RibbonAction} and is isomorphic to $\IJS$. Consequently,
\[ \bP_\alpha \cong \left(\bP_{\alpha^1}\otimes\cdots\otimes \bP_{\alpha^k}\right)
\uparrow\,_{\H_{n_1,\ldots,n_k}(0)}^{\H_n(0)}\ \cong
\bigoplus_{\beta\in[\alpha]}\bP_\beta.\]
\end{theorem}

\begin{proof}
Since (i) is a special case of (ii), it suffices to establish the latter. Sending a standard tableau $\tau$ of shape $\alpha$ to $\pib_{w(\tau)}\pi_{w_0(J\setminus I)}$ gives a vector space isomorphism $\bP_\alpha\cong\IJS$ by  \eqref{eq:BasisIJS}. This isomorphism transforms the $\H_n(0)$-action on $\IJS$ by \eqref{eq:HW0action} to an $\H_n(0)$-action on $\bP_\alpha$ by \eqref{eq:RibbonAction}. Hence $\bP_\alpha$ is a well defined $\H_n(0)$-module isomorphic to $\IJS$. Applying Theorem~\ref{thm:IJS} to $\bP_\alpha\cong\IJS$ completes the proof.
\end{proof}

In particular, a projective indecomposable $\H_n(0)$-module $\bP_\alpha$ has a basis consisting of standard tableaux of ribbon shape $\alpha$, which is in bijection with the descent class of $\alpha$ by taking reading word. Some examples are given below.
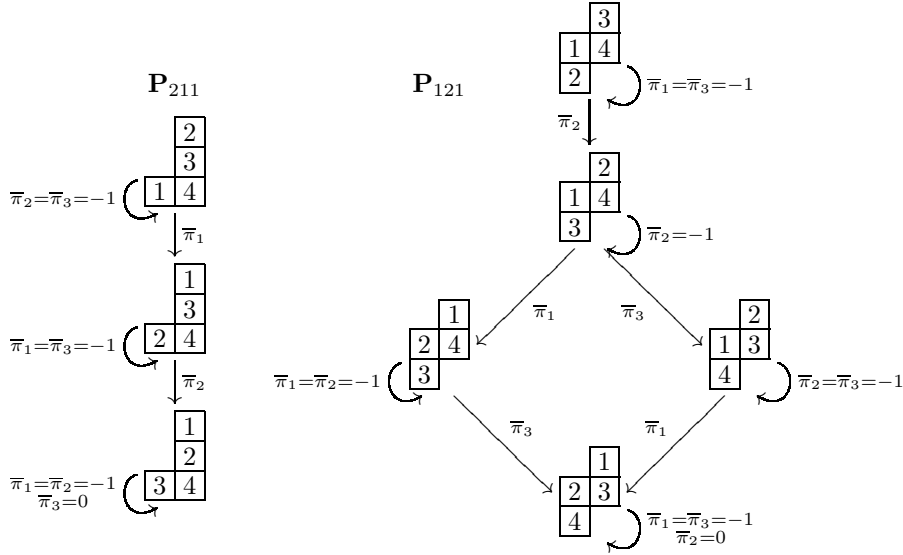
\[ \quad
\xymatrix @R=16pt @C=28pt { \bP_{211} \\
 {\young(:2,:3,14)}  \ar@(ul,dl)[]_{\pib_2=\pib_3=-1} \ar[d]^-{\pib_1}\\
 {\young(:1,:3,24)} \ar@(ul,dl)[]_{\pib_1=\pib_3=-1} \ar[d]^-{\pib_2} \\
{\young(:1,:2,34)} \ar@(ul,dl)[]_{\txt{$\substack{\pib_1=\pib_2=-1 \\ \pib_3=0}$}}  
}\qquad
\xymatrix @R=16pt @C=28pt {
\bP_{121} & {\young(:3,14,2)}  \ar@(ur,dr)[]^{\pib_1=\pib_3=-1} \ar[d]_-{\pib_2}\\
& {\young(:2,14,3)} \ar@(ur,dr)[]^{\pib_2=-1} \ar[ld]^-{\pib_1} \ar[rd]_-{\pib_3} \\
{\young(:1,24,3)} \ar@(ul,dl)[]_{\pib_1=\pib_2=-1} \ar[rd]^-{\pib_3} & &
{\young(:2,13,4)} \ar@(ur,dr)[]^{\pib_2=\pib_3=-1} \ar[ld]_-{\pib_1} \\
& {\young(:1,23,4)} \ar@(ur,dr)[]^{\txt{$\substack{\pib_1=\pib_3=-1\\ \pib_2=0}$}} 
}
\] \vskip5pt

Let $\tau_0(\alpha)$ and $\tau_1(\alpha)$ be the standard tableaux of ribbon shape $\alpha\models n$ whose reading words are $w_0(\alpha)$ and $w_1(\alpha)$, respectively. Then $\tau_0(\alpha)$ is obtained by filling with $1,2,\ldots,n$ the columns of the ribbon $\alpha$ from top to bottom, starting with the leftmost column and proceeding toward the rightmost column. Similarly, $\tau_1(\alpha)$ is obtained by filling with $1,2,\ldots,n$ the rows of the ribbon $\alpha$ from left to right, starting with the top row and proceeding toward the bottom row. One can also check that $\bC_\alpha :=\FF \cdot \tau_1(\rev(\alpha))$ is isomorphic to the top of $\bP_\alpha$, and the reading word of $\tau_1(\rev(\alpha))$ is $w_1(\rev(\alpha))=w_1(\alpha)^{-1}$. The one-dimensional module $\bP_n=\bC_n$ [\,$\bP_{1^n}=\bC_{1^n}$ resp.] admits an $\H_n(0)$-action by $\pib_i=0$ [$\pib_i=1$ rep.] for all $i\in[n-1]$, giving an analogue of the trivial [sign resp.] representation of $\SS_n$. 

\begin{remark}\label{rem:symmetryA}
One sees that $\tau_0(\alpha^t)=\tau_1(\alpha)$ for any composition $\alpha$. In fact, one obtains $\bP_{\alpha^t}$ by taking transpose of standard tableaux of shape $\alpha$, reversing the arrows connecting these tableaux, and modifying loops accordingly. This agrees with the anti-isomorphism $w\mapsto ww_0$ between the intervals $[w_0(\alpha),w_1(\alpha)]$ and $[w_0(\alpha^t),w_1(\alpha^t)]$ in the left weak order of $\SS_n$, giving a representation theoretic interpretation for the antipode of $\NSym$---see Section~\ref{sec:antipodes}. 
\end{remark}

Recall that simple $\CC\SS_n$-modules can be constructed using tabloids. By definition each tabloid can be uniquely represented by a Young tableau with increasing rows. Thus for any composition $\alpha=(\alpha_1,\ldots,\alpha_\ell)\models n$, the standard tableaux of generalized ribbon shape $\alpha_1\oplus\cdots\oplus\alpha_\ell$ are analogous to tabloids, and they form a basis for the $\H_n(0)$-module $\bM_\alpha:=\bP_{\alpha_1\oplus\cdots\oplus\alpha_\ell}$. Moreover, the $\H_n(0)$-module $\bP_\alpha$ is naturally isomorphic to a submodule of $\bM_\alpha$, since one can obtain a standard tableau of shape $\alpha_1\oplus\cdots\oplus\alpha_\ell$ from a standard tableau of ribbon shape $\alpha$ by separating its rows. This gives an analogue for the permutation module $M^\lambda$ and Specht module $S^\lambda$ of the symmetric group $\SS_n$ indexed by a partition $\lambda\vdash n$. See also Example~\ref{ex:M211}. If $\alpha=1^n$ then $\bM_\alpha$ carries the regular representation of $\H_n(0)$. More generally, the following result provides an analogue of Young's rule, which follows immediately from Theorem~\ref{thm:IndP}.

\begin{corollary}
If $\alpha=(\alpha_1,\ldots,\alpha_\ell)\models n$ then 
\[ \bM_\alpha\cong \left( \bP_{\alpha_1} \otimes\cdots \otimes \bP_{\alpha_\ell} \right) \uparrow\,_{\H_\alpha(0)}^{\H_n(0)} \cong \bigoplus_{\beta\cleq\alpha} \bP_\beta.\]
\end{corollary}

\subsection{0-Hecke algebras of type B}\label{sec:H0B}

A \emph{signed permutation} of $[n]$ is a permutation of the set $\{\pm1,\ldots,\pm n\}$ such that $w(-i)=-w(i)$ for all $i\in[\pm n]$. The \emph{hyperoctahedral group} $\B_n$ consists of all signed permutations of $[n]$ with group operation being the composition of maps. Since a signed permutation $w$ in $\B_n$ is determined by where it sends $1,\ldots,n$, we write $w$ as a word $w(1)\cdots w(n)$, where a negative integer $-k<0$ is often written as $\bar k$. We assume $w(0)=0$ for all $w\in\B_n$.

The hyperoctahedral group $\B_n$ is generated by $S=\{s_0,s_1,\ldots,s_{n-1}\}$, where $s_0 := \bar12\cdots n$ and $s_i := (i,i+1)(-i,-(i+1))$ for $i=1,\ldots,n-1$. The pair $(\B_n,S)$ is the Coxeter system of type $B_n$ whose Coxeter diagram is given below.
\[
\xymatrix @C=12pt{
s_0 \ar@{=}[r] & s_1 \ar@{-}[r] & s_2 \ar@{-}[r] & \cdots \ar@{-}[r] & s_{n-2} \ar@{-}[r] & s_{n-1}
}
\]
The symmetric group $\SS_n$ is naturally identified with the parabolic subgroup of $\B_n$ generated by $s_1,\ldots,s_{n-1}$. 

The descent set of $w\in \B_n$ is $D(w)=\{i:0\leq i\leq n-1, w(i)>w(i+1)\}$ where we identify $i$ with $s_i$. 
The length of $w\in \B_n$ is $\ell(w)=\inv(w)+\neg(w)+\nsp(w)$ where
\begin{eqnarray*}
\inv(w) &:=& \#\{(i,j):1\leq i<j\leq n,\ w(i)>w(j)\}, \\
\neg(w) &:=& \#\{1\leq i\leq n: w(i)<0\}, \qand \\
\nsp(w) &:=& \#\{(i,j):1\leq i<j\leq n,\ w(i)+w(j)<0\}.
\end{eqnarray*}

The (complex) representation theory of hyperoctahedral groups is related to the representation theory of symmetric groups and can be found in Geissinger and Kinch~\cite{GeissingerKinch}. The representation theory of 0-Hecke algebras of type B is a special case of the result of Norton~\cite{Norton}, and here we give a combinatorial approach to it using tableaux.

The 0-Hecke algebra $\HB_n(0)$ of the hyperoctahedral group $\B_n$ has two generating sets $\{\pi_i:0\le i\le n-1\}$ and $\{\pib_i:0\le i\le n-1\}$, where $\pib_i := \pi_i-1$. Both generating sets satisfy the same braid relation as the hyperoctahedral group $\B_n$, but different quadratic relations $\pi_i^2=\pi_i$ and $\pib_i^2=-\pib_i$. One can realize $\pi_0,\pi_1,\ldots,\pi_{n-1}$ as the \emph{signed bubble-sorting operators} on $\ZZ^n$: if $0\le i\le n-1$ and $(a_1,\ldots,a_n)\in \ZZ^n$ then
\[ \pi_i(a_1,\ldots,a_n):=
\begin{cases}
(-a_1,a_2,\ldots, a_n), & i=0,\ a_1>0,\\
(a_1,\ldots,a_{i+1},a_i,\ldots,a_n), & 1\leq i\leq n-1,\ a_i<a_{i+1},\\
(a_1,\ldots,a_n), & {\rm otherwise}.
\end{cases} \]

We need some notation before presenting our tableau approach to the representation theory of $\HB_n(0)$. A \emph{pseudo-composition} is a sequence $\alpha=(\alpha_1,\ldots,\alpha_\ell)$ of integers such that $\alpha_1\geq0$ and $\alpha_2,\ldots,\alpha_\ell>0$. The \emph{length} of $\alpha$ is $\ell(\alpha):=\ell$ and the \emph{size} of $\alpha$ is $|\alpha|:=\alpha_1+\cdots+\alpha_\ell$. We call $\alpha$ a pseudo-composition of $n$ and write $\alpha\modelsB n$ if its size is $n$. The \emph{descent set} of $\alpha$ is $D(\alpha):=\{\alpha_1,\alpha_1+\alpha_2,\ldots,\alpha_1+\cdots+\alpha_{\ell-1}\}$.
The map $\alpha\mapsto D(\alpha)$ is a bijection between pseudo-compositions of $n$ and the subsets of $\{0,1,\ldots,n-1\}$.
Let $\alpha^c$ be the pseudo-composition of $n$ whose descent set is $\{0,1,\ldots,n-1\} \setminus D(\alpha)$. 
We write $\alpha\cleq\beta$ if $\alpha$ and $\beta$ are pseudo-compositions of the same size such that $D(\alpha)\subseteq D(\beta)$.
A pseudo-composition $\alpha=(\alpha_1,\ldots,\alpha_\ell)$ and a composition $\beta=(\beta_1,\ldots,\beta_k)$ give rise to two pseudo-compositions
\[ \alpha\cdot\beta := (\alpha_1,\ldots,\alpha_\ell,\beta_1,\ldots,\beta_k) \qand 
\alpha\rhd\beta := (\alpha_1,\ldots,\alpha_{\ell-1},\alpha_\ell+\beta_1,\beta_2,\ldots,\beta_k). \]

Let $\alpha=(\alpha_1,\ldots,\alpha_\ell)\modelsB n$. The \emph{parabolic subgroup} $\B_\alpha$ of $\B_n$ is generated by $\{s_i:i\in D(\alpha^c)\}$ and isomorphic to $\B_{\alpha_1}\times\SS_{\alpha_2}\times \cdots\times\SS_{\alpha_\ell}$.  The \emph{parabolic subalgebra} $\HB_\alpha(0)$ of $\HB_n(0)$ is generated by $\{\pi_i: i\in D(\alpha^c)\}$ and isomorphic to $\HB_{\alpha_1}(0) \otimes\H_{\alpha_2}(0)\otimes\cdots\otimes \H_{\alpha_\ell}(0)$. The \emph{descent class of $\alpha$ in $\B_n$} is the set $\{w\in\B_n: D(w) = D(\alpha) \}$, which is an interval under the left weak order of $\B_n$, denoted by $[w^B_0(\alpha),w^B_1(\alpha)]$.

Specializing Norton's results to type B one sees that projective indecomposable $\HB_n(0)$-modules and simple $\HB_n(0)$-modules are given by $\P_I^S$ and $\C_I^S$ for all $I\subseteq S = \{s_0,s_1,\ldots,s_{n-1}\}$. These modules can be indexed by pseudo-compositions of $n$ whose descent sets correspond to subsets of $S$. As discussed in Section~\ref{sec:HW0}, there are $\HB_n(0)$-modules $\IJS$ for all pairs of subsets $I,J\in S$, which generalize the projective indecomposable $\HB_n(0)$-modules $\P_I^S$. We will describe all these $\HB_n(0)$-modules using tableaux.

First note that pseudo-compositions are in bijection with \emph{pseudo-ribbon diagrams} in the following way. If a pseudo-composition does not begin with a $0$ then it can be viewed as a composition which corresponds to a ribbon diagram, and we add a new 0-box to the left of the bottom row of this ribbon diagram. If a pseudo-composition begins with a $0$ followed by a composition $\alpha$, then we draw the ribbon diagram corresponding to $\alpha$ and add a new $0$-box below the leftmost column of this ribbon diagram. For example, the pseudo-compositions $(2,3,1,1)$ and $(0,2,3,1,1)$ correspond to the following pseudo-ribbon diagrams.
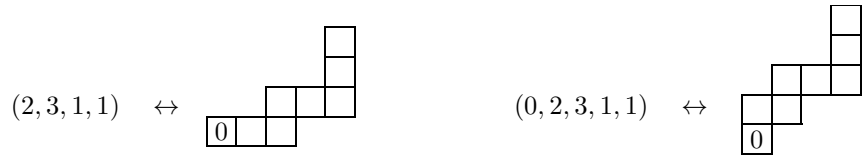
\[ {(2,3,1,1) \quad \leftrightarrow \quad} \raisebox{-12pt}{\young(::::\hfill,::::\hfill,::\hfill\hfill\hfill,0\hfill\hfill)}
\qquad\qquad\qquad 
(0,2,3,1,1) \quad \leftrightarrow \quad \raisebox{-15pt}{\young(:::\hfill,:::\hfill,:\hfill\hfill\hfill,\hfill\hfill,0)} \] 
If $\alpha$ is a pseudo-composition then the pseudo-ribbon diagrams of $\alpha$ and $\alpha^c$ are symmetric about the $45^\circ$-diagonal. 

A \emph{generalized pseudo-ribbon} $\alpha=\alpha^1\oplus\cdots\oplus\alpha^k$ is a skew diagram with connected components $\alpha^1,\ldots,\alpha^k$ such that $\alpha^1$ is a pseudo ribbon, $\alpha^2,\ldots,\alpha^k$ are ribbons, and $\alpha^{i+1}$ is strictly to the northeast of $\alpha^i$ for $i=1,\ldots,k-1$. The size of $\alpha$ is $|\alpha|:=|\alpha^1|+\cdots+|\alpha^k|$. Define $[\alpha^1]:=\{\alpha^1\}$ and 
\[ [\alpha] := \left\{ \beta \cdot \alpha_k,\ \beta\rhd \alpha_k : \beta\in [\alpha^1\oplus\cdots\oplus\alpha^{k-1}] \right\}. \]

Let $\alpha$ be a generalized pseudo-ribbon. A \emph{(type B) standard tableau $\tau$ of shape} $\alpha$ is a filling of the empty boxes in $\alpha$ with integers $\pm1,\ldots,\pm n$, such that each row is increasing from left to right and each column is increasing from top to bottom, \emph{with the extra 0-box included}, and that the \emph{reading word} $w(\tau)$ of $\tau$ is a signed permutation in $\B_n$. Here the reading word $w(\tau)$ is obtained by reading the integers in $\tau$, \emph{with the extra 0 excluded}, from the bottom row to the top row and proceeding from left to right within each row. Below are two standard tableaux of ribbon shape $2311$ and $02311$ with reading word $23\bar4\bar16\bar5\bar7$ and $\bar65\bar4172\bar3$, respectively.

\[ \ytableaushort{ \none\none\none\none{\bar7}, \none\none\none\none{\bar5}, \none\none {\bar4}{\bar1}6,{\color{red}0}23 } \qquad\qquad \qquad\qquad
\ytableaushort{ \none\none\none{\bar3}, \none\none\none{2}, \none {\bar4}17,{\bar6}5,{\color{red}0} } \] 

One can negate every entry of a standard tableau $\tau$ of a pseudo-ribbon shape, rotate it by 180 degrees, and glue it back to $\tau$ by identifying the two extra 0-boxes. The resulting tableau has reading word given by the full description for the signed permutation $w(\tau)$. See an example below.

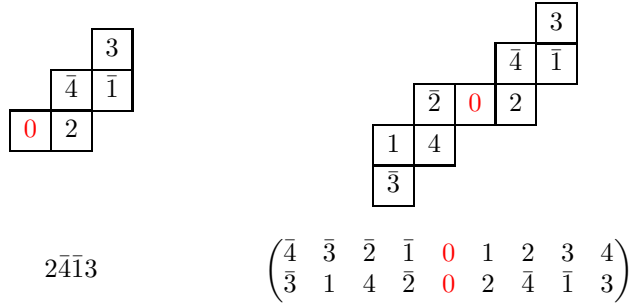
\[ \begin{matrix}
\raisebox{-10pt}{\ytableaushort{\none\none 3, \none { \bar4}{\bar1}, {\color{red}0}2 }} & \qquad\qquad  \qquad
\ytableaushort{\none\none\none\none 3, \none\none\none { \bar4}{\bar1}, \none{\bar2} {{\color{red}0}} 2, 14,{\bar3} } \\ \\ 2\bar4\bar13 & \qquad\qquad
\begin{pmatrix} 
\bar4 & \bar 3 & \bar 2 & \bar1 & {{\color{red}0}} & 1 & 2 &3 & 4 \\
\bar3 & 1 & 4 & \bar2 & {{\color{red}0}} &2 &\bar4 & \bar1 &3
\end{pmatrix}
\end{matrix} \]\vskip5pt

Let $\alpha$ be a generalized ribbon of size $n$. One sees that taking reading word gives a bijection between standard tableaux of shape $\alpha$ and the union of descent classes of $\beta$ in $\B_n$ for all $\beta \in [\alpha]$. We define $\bPB_\alpha$ to be the vector space with a basis consisting of all standard tableaux of shape $\alpha$. Let $\tau$ be a tableau in this basis. Define $s_0(\tau)$ to be the tableau obtained from $\tau$ by negating the entry $\pm1$, and call $0$ a \emph{descent} of $\tau$ if $-1$ appears in $\tau$. For every $i\in[n-1]$ define $s_i(\tau)$ to be the tableau obtained from $\tau$ by swapping the absolute values of $\pm i$ and $\pm(i+1)$ but leaving their signs invariant, and call $i$ a \emph{descent} of $\tau$ if one of the following conditions holds:
\begin{itemize}
\item
both $i$ and $i+1$ appear in $\tau$, with $i$ in a row higher than $i+1$,
\item
both $-i$ and $-(i+1)$ appear in $\tau$, with $-i$ in a row lower than $-(i+1)$,
\item
both $i$ and $-(i+1)$ appear in $\tau$.
\end{itemize}
One sees that the descents of $\tau$ are precisely the descents of $w(\tau)^{-1}$. For each $i\in\{0,1,\ldots,n-1\}$ we define
\begin{equation}\label{eq:ActionB}
\pib_i(\tau)=
\begin{cases}
-\tau, & \textrm{ if $i$ is a descent of $\tau$},\\
0, & \textrm{ if $i$ is not a descent of $\tau$ and $s_i(\tau)$ is not standard},\\
s_i(\tau), & \textrm{ if $i$ is not a descent of $\tau$ and $s_i(\tau)$ is standard}.
\end{cases}
\end{equation}

\begin{theorem}\label{thm:IndPB}
(i) If $\alpha\modelsB n$ then $\bPB_\alpha \cong \P_I^S$ where $I=\{s_i:i\in D(\alpha)\}$.

\noindent(ii) Let $\alpha=\alpha^1 \oplus \cdots \oplus \alpha^k$ be a generalized pseudo-ribbon of size $n$ with connected components $\alpha^1\modelsB n_1$ and $\alpha^i\models n_i$ for $i=2,\ldots,k$. Let 
\[ I = \{s_i: i\in D(\alpha^1\rhd\cdots\rhd \alpha^k)\} \qand
J = \{s_j:j\in D(0\cdot1^{n_1}\rhd\cdots\rhd 1^{n_k}) \}.\]
 Then $\bPB_\alpha$ is a well defined $\HB_n(0)$-module isomorphic to $\IJS$. Consequently,
\[ \bPB_\alpha \cong \left(\bPB_{\alpha^1}\otimes\bP_{\alpha_2}\otimes\cdots\otimes \bP_{\alpha^k}\right)
\uparrow\,_{\HB_{n_1,\ldots,n_k}(0)}^{\HB_n(0)}\ \cong
\bigoplus_{\beta\in[\alpha]}\bPB_\beta. \]
\end{theorem}

\begin{proof}
This can be proved similarly to Theorem~\ref{thm:IndP}.
\end{proof}

In particular, a projective indecomposable $\HB_n(0)$-module $\bPB_\alpha$ has a basis consisting of standard tableaux of pseudo-ribbon shape $\alpha$, which is in bijection with the descent class of $\alpha$ in $\B_n$. See the example below.

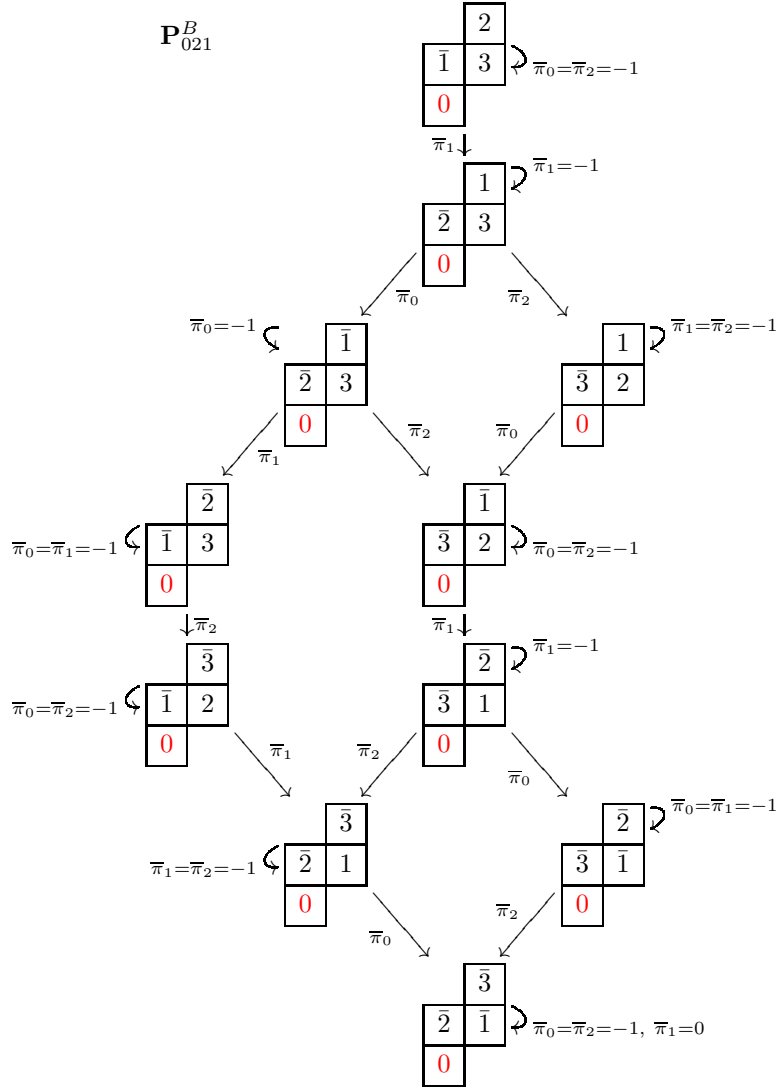
\[ 
\xymatrix @R=8pt @C=16pt { 
 \bPB_{021} & &
{\begin{ytableau}
   \none & 2 \\
   \bar1 & 3 \\
    {\color{red}0}
  \end{ytableau}}
 \ar@(r,dr)[]^{\pib_0=\pib_2=-1} \ar[d]_-{\pib_1}  \\
& & 
{\begin{ytableau}
   \none & 1 \\
   \bar2 & 3 \\
    {\color{red}0}
  \end{ytableau}}
\ar@(ur,r)[]^{\pib_1=-1} \ar[ld]^-{\pib_0} \ar[rd]_-{\pib_2} \\
& 
{\begin{ytableau}
   \none & \bar1 \\
   \bar2 & 3 \\
    {\color{red}0}
  \end{ytableau}}
\ar@(ul,l)[]_{\pib_0=-1} \ar[ld]^-{\pib_1} \ar[rd]^-{\pib_2} & &
{\begin{ytableau}
   \none & 1 \\
   \bar3 & 2 \\
    {\color{red}0}
  \end{ytableau}}
\ar@(ur,r)[]^{\pib_1=\pib_2=-1} \ar[ld]_-{\pib_0} \\
{\begin{ytableau}
   \none & \bar2 \\
   \bar1 & 3 \\
    {\color{red}0}
  \end{ytableau}}
\ar@(l,dl)[]_{\pib_0=\pib_1=-1} \ar[d]^-{\pib_2} & &
{\begin{ytableau}
   \none & \bar1 \\
   \bar3 & 2\\
    {\color{red}0}
  \end{ytableau}}
\ar@(r,dr)[]^{\pib_0=\pib_2=-1} \ar[d]_-{\pib_1} \\
{\begin{ytableau}
   \none & \bar3 \\
   \bar1 & 2 \\
    {\color{red}0}
  \end{ytableau}}
\ar@(l,dl)[]_{\pib_0=\pib_2=-1} \ar[rd]^-{\pib_1}
& & 
{\begin{ytableau}
   \none & \bar2 \\
   \bar3 & 1 \\
    {\color{red}0}
  \end{ytableau}} 
\ar@(ur,r)[]^{\pib_1=-1} \ar[ld]_-{\pib_2} \ar[rd]_-{\pib_0}
\\
& 
{\begin{ytableau}
   \none & \bar3 \\
   \bar2 & 1 \\
    {\color{red}0}
  \end{ytableau}}
\ar@(l,dl)[]_{\pib_1=\pib_2=-1} \ar[rd]_-{\pib_0}
& & 
{\begin{ytableau}
   \none & \bar2 \\
   \bar3 & \bar1 \\
    {\color{red}0}
  \end{ytableau}}
\ar@(ur,r)[]^{\pib_0=\pib_1=-1} \ar[ld]_-{\pib_2} \\
& &
{\begin{ytableau}
   \none & \bar3 \\
   \bar2 & \bar1 \\
    {\color{red}0}
  \end{ytableau}}
\ar@(r,dr)[]^{\pib_0=\pib_2=-1,\ \pib_1=0}
} \] \vskip5pt

Let $\alpha$ be a pseudo-composition of $n$. Denote by $\tau^B_0(\alpha)$ and $\tau^B_1(\alpha)$ the standard tableaux of shape $\alpha$ whose reading words are $w^B_0(\alpha)$ and $w^B_1(\alpha)$, respectively. They can be constructed in the following way.

Denote by $\tau_0$ the tableau obtained by filling with $\bar1,\bar2,\ldots,\bar c$ the empty boxes (if any) on the leftmost column of $\alpha$ from bottom to top, where $c$ is the number of these empty boxes, and then filling with $c+1,c+2,\ldots,n$ the remaining columns of $\alpha$ from top to bottom, starting from the second leftmost column and proceeding toward the rightmost column. One can check that $\pib_i(\tau_0)=-\tau_0$ if $i\in D(\alpha)$ and $\pib_i(\tau_0)=s_i(\tau_0)$ otherwise. Hence $\tau_0=\tau^B_0(\alpha)$. 

Similarly, let $\tau_1$ be the tableau obtained by filling with $1,2,\ldots,r$ the empty boxes (if any) on the bottom row of $\alpha$ from left to right, where $r$ is the number of these empty boxes, and then filling with $-(r+1),-(r+2),\ldots,-n$ the remaining rows of $\alpha$ from right to left, proceeding from the second bottom row toward the top row. One can check that $\pib_i(\tau_1)=-\tau_1$ if $i\in D(\alpha)$ and $\pib_i(\tau_1)=0$ otherwise. Hence $\tau_1=\tau^B_1(\alpha)$ 
and $\bCB_\alpha:=\FF\cdot\tau_1$ is isomorphic to the top of $\bPB_\alpha$.

Given a standard tableau $\tau$ of pseudo-ribbon shape $\alpha$, define $\theta^B(\tau)$ to be the standard tableau of shape $\alpha^c$ obtained by reflecting $\tau$ across the $45^\circ$-diagonal and negating every entry in it. The above construction implies  $\tau^B_1(\alpha^c)=\theta^B(\tau^B_0(\alpha))$. Some examples are given below.  \vskip3pt
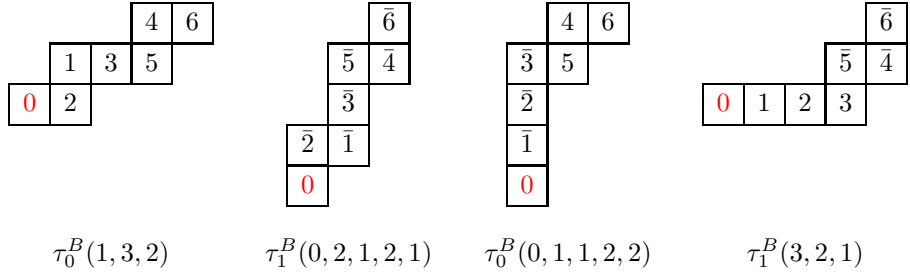
\[ \begin{matrix}
\ytableaushort{\none\none\none46,\none135,{\color{red}0}2} &  \quad
\ytableaushort{\none\none{\bar6},\none{\bar5}{\bar4},\none{\bar3},{\bar2}{\bar1},{\color{red}0}} &  \quad
\ytableaushort{\none46,{\bar3}5,{\bar2},{\bar1},{\color{red}0}} &  \quad
\ytableaushort{\none\none\none\none{\bar6},\none\none\none{\bar5}{\bar4},{\color{red}0}123}  \\ \\
\tau^B_0(1,3,2) &  \quad \tau^B_1(0,2,1,2,1) & \quad \tau^B_0(0,1,1,2,2) &  \quad \tau^B_1(3,2,1)
\end{matrix}  \] \vskip5pt
\noindent In fact, one can obtain $\bPB_{\alpha^c}$ from $\bPB_{\alpha}$ by applying $\theta^B$ to standard tableaux of shape $\alpha$, reversing arrows connecting these tableaux, and modifying loops accordingly. This agrees with the anti-isomorphism $w\mapsto w_0w$ between the intervals $[w^B_0(\alpha),w^B_1(\alpha)]$ and $[w^B_0(\alpha^c),w^B_1(\alpha^c)]$ in the left weak order of $\B_n$, where $w_0$ is the longest element of $\B_n$. A uniform approach to this for 0-Hecke algebras will be discussed in Section~\ref{sec:antipodes}.

Lastly, we generalize our analogue of Young's rule from type A to type B. Let $\alpha=(\alpha_1,\ldots,\alpha_\ell)$ be a pseudo-composition and define $\bMB_\alpha := \bPB_{\alpha_1\oplus\cdots\oplus\alpha_\ell}$. Then $\bPB_\alpha$ is naturally isomorphic to a submodule of $\bMB_\alpha$, since one can obtain a standard tableau of shape $\alpha_1\oplus\cdots\oplus\alpha_\ell$ from a standard tableau of  shape $\alpha$ by separating its rows. If $\alpha=0\cdot 1^n$ then $\bMB_\alpha$ carries the regular representation of $\HB_n(0)$. More generally, one has the following corollary of Theorem~\ref{thm:IndPB}.

\begin{corollary}
Let $\alpha=(\alpha_1,\alpha_2,\ldots,\alpha_\ell)$ be a pseudo-composition of $n$. Then 
\[ \bMB_\alpha\cong \left( \bPB_{\alpha_1} \otimes \bP_{\alpha_2} \otimes\cdots \otimes \bP_{\alpha_\ell} \right) \uparrow\,_{\HB_\alpha(0)}^{\HB_n(0)} \cong \bigoplus_{\beta\cleq\alpha} \bPB_\beta. \]
\end{corollary}

\subsection{0-Hecke algebras of type D}\label{sec:H0D}

Assume $n\ge 2$. The hyperoctahedral group $\B_n$ admits a subgroup $\D_n$ consisting of all signed permutations $w\in\B_n$ with $\neg(w)$ even. It is generated by $\{s_0,s_1,\ldots,s_{n-1}\}$, where $s_0:=\bar2\bar13\cdots n$ and $s_i:=(i,i+1)(-i,-(i+1))$ for all $i\in[n-1]$. This gives the finite irreducible Coxeter system of type $D_n$ whose Coxeter diagram is below.
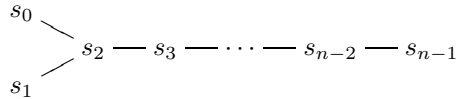
\[ \xymatrix @R=2pt @C=12pt{
s_0 \ar@{-}[rd] \\
& s_2 \ar@{-}[r] & s_3 \ar@{-}[r] & \cdots \ar@{-}[r] & s_{n-2} \ar@{-}[r] & s_{n-1} \\
s_1 \ar@{-}[ru]
} \]
Subsets of $\{s_0,s_1,\ldots,s_{n-1}\}$ are still indexed by pseudo-compositions $\alpha$ of $n$, which we denote by $\alpha\modelsD n$ for consistency of notation.
The descent set of $w\in\D_n$ is $D(w)=\{i: 0\le i\le n-1, w(i)>w(i+1)\}$ and the length of $w\in\D_n$ is $\ell(w)=\inv(w)+\nsp(w)$, where $w(0):=-w(2)$.

The 0-Hecke algebra $\HD_n(0)$ of $\D_n$ has generating sets $\{\pi_0,\pi_1,\ldots,\pi_{n-1}\}$ and $\{\pib_0,\pib_1,\ldots,\pib_{n-1}\}$, both satisfying the same braid relations as $\D_n$ but different quadratic relations $\pi_i^2=\pi_i$ and $\pib_i^2=-\pib_i$. One can realize $\pi_i$ as operators on $\ZZ^n$: if $0\le i\le n-1$ and $a=(a_1,\ldots,a_n)\in\ZZ^n$ then 
\[ \pi_i(a_1,\ldots,a_n):=
\begin{cases}
(-a_2,-a_1,a_3,\ldots, a_n), & \text{if } i=0,\ a_1+a_2>0,\\
(a_1,\ldots,a_{i+1},a_i,\ldots,a_n), & \text{if } 1\leq i\leq n-1,\ a_i<a_{i+1},\\
(a_1,\ldots,a_n), & \text{otherwise}.
\end{cases} \] 

Let $\alpha=(\alpha_1,\alpha_2,\ldots,\alpha_\ell)\modelsD n\ge2$. The \emph{parabolic subgroup} $\D_\alpha$ of $\D_n$ is generated by $\{s_i:i\in D(\alpha^c)\}$, and the \emph{parabolic subalgebra} $\HD_\alpha(0)$ of $\HD_n(0)$ is generated by $\{\pi_i: i\in D(\alpha^c)\}$. If $\alpha_1\ge2$ then 
\[ \D_\alpha \cong \D_{\alpha_1}\times\SS_{\alpha_2}\times \cdots\times\SS_{\alpha_\ell} \qand
\HD_\alpha(0) \cong \HD_{\alpha_1}(0) \otimes\H_{\alpha_2}(0)\otimes\cdots\otimes \H_{\alpha_\ell}(0).\]
The \emph{descent class of $\alpha$ in $\D_n$} is the set $\{w\in\D_n: D(w) = D(\alpha) \}$, which is an interval under the left weak order of $\D_n$, denoted by $[w_0^D(\alpha),w_1^D(\alpha)]$. 


Now let $\alpha$ be a generalized pseudo-ribbon of size $n\ge2$. A \emph{type D standard tableau $\tau$ of shape $\alpha$} is a filling  of the empty boxes in $\alpha$ with $\pm1,\pm2,\ldots,\pm n$, such that each row is increasing from left to right and each column is increasing from top to bottom, \emph{with the extra 0-entry interpreted as $-w(2)$}, and that the \emph{reading word} $w(\tau)$ belongs to $\D_n$. Here the reading word $w(\tau)$ is obtained by reading the integers in $\tau$ in the same way as in type B, again \emph{excluding the extra 0-entry}. Taking reading word gives a bijection between type D standard tableaux of shape $\alpha$ and the union of the descent classes of $\beta$ in $\D_n$ for all $\beta\in[\alpha]$. We define the \emph{descents} of $\tau$ to be the descents of $w(\tau)^{-1}$. One sees that $0$ is a descent of $\tau$ if and only if one of the following conditions holds:
\begin{itemize}
\item
both $\bar 1$ and $\bar2$ appear in $\tau$, 
\item
both $-1$ and $2$ appear in $\tau$, with $-1$ on a higher row than $2$,
\item
both $1$ and $-2$ appear in $\tau$, with $-2$ on a higher row than $1$,
\end{itemize}
and a positive integer $i\in[n-1]$ is a descent of $\tau$ if and only if one of the following conditions holds:
\begin{itemize}
\item
both $i$ and $i+1$ appear in $\tau$, with $i$ on a row higher than $i+1$,
\item
both $-i$ and $-(i+1)$ appear in $\tau$, with $-i$ on a row lower than $-(i+1)$,
\item
both $i$ and $-(i+1)$ appear in $\tau$.
\end{itemize}

We denote by $\bPD_\alpha$ the vector space with a basis consisting of all type D standard tableaux of shape $\alpha$. Let $\tau$ be a tableau in this basis. Define $s_0(\tau)$ to be the tableau obtained by swapping the absolute values of $\pm1$ and $\pm2$ in $\tau$ and negating the original signs at these two positions. Define $s_i(\tau)$ to be the tableau obtained by swapping the absolute values of $\pm i$ and $\pm(i+1)$ in $\tau$ but leaving their signs invariant. For every $i\in\{0,1,2,\ldots,n-1\}$ we define
\[ \pib_i(\tau)=
\begin{cases}
-\tau, & \textrm{ if $i\in D(w(\tau)^{-1})$},\\
0, & \textrm{ if $i\notin D(w(\tau)^{-1})$ and $s_i(\tau)$ is not type D standard},\\
s_i(\tau), & \textrm{ if  $i\notin D(w(\tau)^{-1})$ and $s_i(\tau)$ is type D standard}.
\end{cases} \]

\begin{theorem}\label{thm:IndPD}
(i) If $\alpha\modelsD n\ge2$ then $\bPD_\alpha \cong \P_I^S$ where $I=\{s_i:i\in D(\alpha)\}$.

\noindent(ii) Let $\alpha=\alpha^1 \oplus \cdots \oplus \alpha^k$ be a generalized pseudo-ribbon of size $n\ge2$. Let 
\[ I = \{s_i: i\in D(\alpha^1\rhd\cdots\rhd \alpha^k)\} \qand
J = \{s_j:j\in D(0\cdot1^{n_1}\rhd\cdots\rhd 1^{n_k}) \}.\]
Then \[ \bPD_\alpha \cong \IJS \cong \bigoplus_{\beta\in[\alpha]}\bPD_\beta. \]
If in addition $\alpha^1\modelsD n_1\ge2$ and $\alpha^i \models n_i$ for $i=2,\ldots,k$ then 
\[\bPD_\alpha \cong \left(\bPD_{\alpha^1}\otimes\bP_{\alpha_2}\otimes\cdots\otimes \bP_{\alpha^k}\right) \uparrow\,_{\HD_{n_1,\ldots,n_k}(0)}^{\HD_n(0)}.\]
\end{theorem}

\begin{proof}
This can be proved similarly to Theorem~\ref{thm:IndP}.
\end{proof}

In particular, a projective indecomposable $\HD_n(0)$-module $\bPD_\alpha$ has a basis consisting of type D standard tableaux of shape $\alpha\modelsD n$, which is in bijection with the descent class $[w^D_0(\alpha),w^D_1(\alpha)]$ of $\alpha$ in $\D_n$ by taking reading word. See the following example.

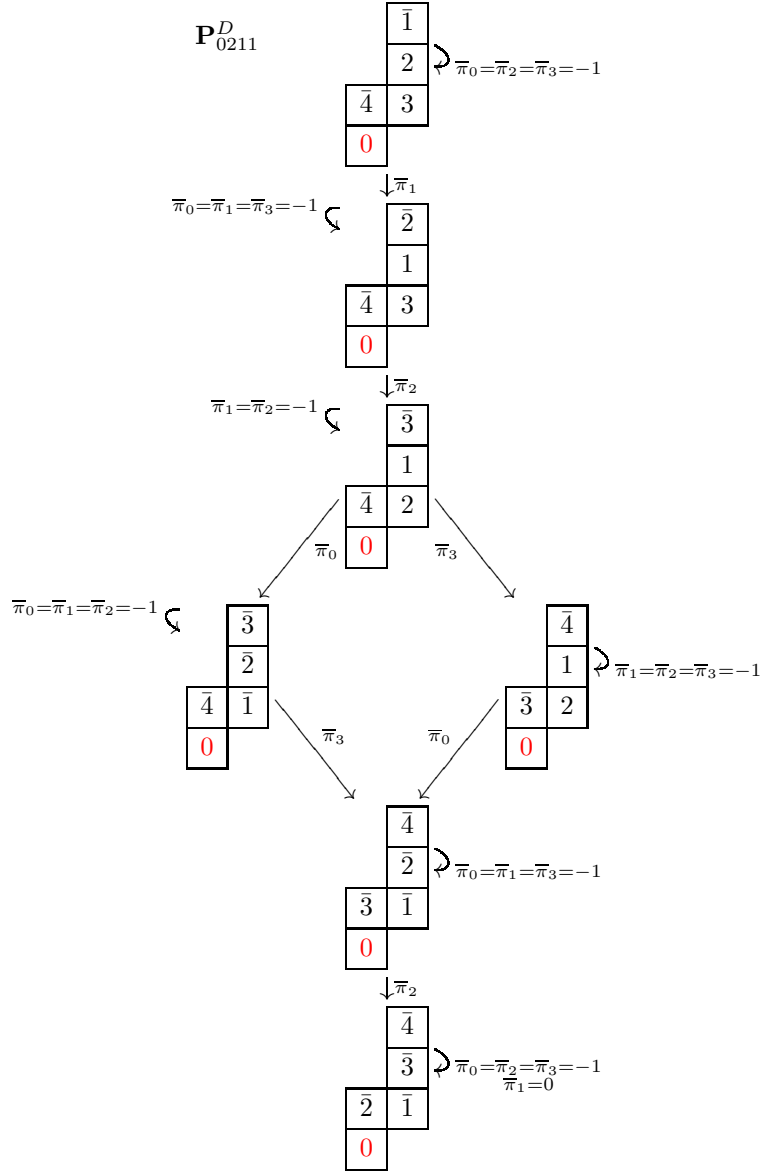
\[  \xymatrix @R=8pt{ \bPD_{0211} 
& \ytableaushort{\none{\bar1},\none2,{\bar4}3,{\color{red}0}} 
\ar@(r,rd)[]^{\pib_0=\pib_2=\pib_3=-1} \ar[d]^-{\pib_1} \\
& \ytableaushort{\none{\bar2},\none1,{\bar4}3,{\color{red}0}} 
\ar@(lu,l)[]_{\substack{\pib_0=\pib_1=\pib_3=-1}} \ar[d]^-{\pib_2} \\
& \ytableaushort{\none{\bar3},\none1,{\bar4}2,{\color{red}0}} \ar@(lu,l)[]_{\substack{\pib_1=\pib_2=-1}} 
\ar[ld]^{\pib_0} \ar[rd]_{\pib_3} \\
\ytableaushort{\none{\bar3},\none{\bar2},{\bar4}{\bar1},{\color{red}0}} 
\ar@(lu,l)[]_{\substack{\pib_0=\pib_1=\pib_2=-1}} \ar[rd]^-{\pib_3} & &
\ytableaushort{\none{\bar4},\none1,{\bar3}2,{\color{red}0}}
\ar@(r,rd)[]^{\substack{\pib_1=\pib_2=\pib_3=-1}} \ar[ld]_-{\pib_0} \\
& \ytableaushort{\none{\bar4},\none{\bar2},{\bar3}{\bar1},{\color{red}0}}
\ar@(r,rd)[]^{\substack{\pib_0=\pib_1=\pib_3=-1}} \ar[d]^-{\pib_2} \\
& \ytableaushort{\none{\bar4},\none{\bar3},{\bar2}{\bar1},{\color{red}0}} 
\ar@(r,rd)[]^{\substack{\pib_0=\pib_2=\pib_3=-1\\ \pib_1=0}}  }
\]\vskip5pt 

The type D standard tableaux $\tau^D_0(\alpha)$ and $\tau^D_1(\alpha)$ corresponding to $w^D_0(\alpha)$ and $w^D_1(\alpha)$ can be obtained in the following way. Suppose that $c_1$ and $c_2$ are the numbers of empty boxes in the leftmost two columns of $\alpha$.  If $c_1\ne1$ and $\tau^B_0(\alpha)$ has an even number of signs then let $\tau_0=\tau^B_0(\alpha)$. If $c_1\ne1$ and $\tau^B_0(\alpha)$ has an odd number of signs then let $\tau_0$ be the tableau obtained from $\tau^B_0(\alpha)$ by negating the sign of $\pm1$. If $c_1=1$ then let $\tau_0$ be the tableau obtained by filling the second leftmost column of $\alpha$ with $-1,2,\ldots,c_2$ from top to bottom, then filling the leftmost column with $-c_2-1$, and then filling the remaining columns with $c_2+2,\ldots,n$ from top to bottom, proceeding from the third leftmost column to the rightmost column. One can check that $\pib_i(\tau_0)=-\tau_0$ for all $i\in D(\alpha)$ and $\pib_i(\tau_0)=s_i(\tau)$ for all $i\notin D(\alpha)$. Hence $\tau_0=\tau^D_0(\alpha)$.

Similarly, let $r_1$ and $r_2$ be the number of empty boxes on the bottom two rows of $\alpha$. If $r_1\ne1$ and $\tau^B_1(\alpha)$ has an even number of signs then let $\tau_1=\tau^B_1(\alpha)$. If $r_1\ne1$ and $\tau^B_1(\alpha)$ has an odd number of signs then let $\tau_1$ be the tableau obtained from $\tau^B_1(\alpha)$ by negating the sign of $\pm1$. If $r_1=1$ then let $\tau_1$ be the tableau obtained by filling the second bottom column of $\alpha$ with $(-1)^n,-2,\ldots,-r_2$ from right to left, then filling the bottom column with $r_2+1$, and then filling the remaining columns with $-r_2-2,\ldots,-n$ from right to left, proceeding from the third bottom column to the top column. One can check that $\pib_i(\tau_1)=-\tau_1$ for all $i\in D(\alpha)$ and $\pib_i(\tau_1)=0$ for all $i\notin D(\alpha)$.  Hence $\tau_1=\tau^D_1(\alpha)$ and $\bCD_\alpha:=\FF \tau_1$ is isomorphic to the top of $\bPD_\alpha$. Some examples are provided below.
\vskip5pt\[ \begin{matrix}
\ytableaushort{\none\none\none46,\none135,{\color{red}0}2} &\qquad 
\ytableaushort{\none46,{\bar3}5,{\bar2},{1},{\color{red}0}} &\qquad
\ytableaushort{\none\none\none5,\none{\bar1}46,{\bar3}2,{\color{red}0}} \\ \\
\tau^D_0(1,3,2)&\qquad \tau^D_0(0,1,1,2,2)  &\qquad \tau^D_0(0,2,3,1) \\ \\
\ytableaushort{\none\none{\bar6},\none{\bar5}{\bar4},\none{\bar3},{\bar2}{\bar1},{\color{red}0}} &\qquad
\ytableaushort{\none\none\none\none{\bar6},\none\none\none{\bar5}{\bar4},{\color{red}0}{\bar1}23} &\qquad
\ytableaushort{\none\none{\bar6}{\bar5},\none\none{\bar4},\none{\bar2}1,{\color{red}0}3}  \\  \\
\tau^D_1(0,2,1,2,1) &\qquad \tau^D_1(3,2,1) &\qquad \tau^D_1(1,2,1,2)
\end{matrix} \]\vskip5pt

Let $\tau$ be a type D standard tableau of pseudo-ribbon shape $\alpha$. Recall that $\theta^B(\tau)$ is obtained by reflecting $\tau$ across the $45^\circ$-diagonal and negating every entry in it. If $\theta^B(\tau)$ contains an even number of signs then define $\theta^D(\tau)=\theta^B(\tau)$; otherwise define $\theta^D(\tau)$ to be the tableau obtained by negating $\pm1$ in $\theta^B(\tau)$. One sees that $\theta^D(\tau)$ is a type D standard tableau of shape $\alpha^c$. The above construction implies that $\tau^D_1(\alpha^c)=\theta^D(\tau^D_0(\alpha))$. In fact, one can obtain $\bPD_{\alpha^c}$ from $\bPD_{\alpha}$ by applying $\theta^D$ to standard tableaux of shape $\alpha$, reversing arrows connecting these tableaux, and modifying loops accordingly. This agrees with the anti-isomorphism $w\mapsto w_0w$ between the intervals $[w^D_0(\alpha),w^D_1(\alpha)]$ and $[w^D_0(\alpha^c),w^D_1(\alpha^c)]$ in the weak order of $\D_n$, where $w_0$ is the longest element of $\D_n$. See the two examples of $\bPD_{0211}$ and $\bPD_{13}$ provided earlier.  A uniform approach to this for 0-Hecke algebras will be discussed in Section~\ref{sec:antipodes}.

Lastly, let $\alpha=(\alpha_1,\ldots,\alpha_\ell)$ be a pseudo-composition of $n$ and define $\bMD_\alpha := \bPD_{\alpha_1\oplus\cdots\oplus\alpha_\ell}$. Then $\bPD_\alpha$ is naturally isomorphic to a submodule of $\bMD_\alpha$. If $\alpha=0\cdot 1^n$ then $\bMD_\alpha$ carries the regular representation of $\HD_n(0)$. More generally, one has the following corollary of Theorem~\ref{thm:IndPD}.

\begin{corollary}
Let $\alpha=(\alpha_1,\alpha_2,\ldots,\alpha_\ell)\modelsD n\ge2$. Then 
\[ \bMD_\alpha 
\cong \bigoplus_{\beta\cleq\alpha} \bPD_\beta. \]
\end{corollary}

\section{Quasisymmetric functions and noncommutative symmetric functions}\label{sec:Char}

In this section we investigate connections between the representation theory of 0-Hecke algebras of type A, B, and D and quasisymmetric functions and noncommutative symmetric functions of type A, B, and D.

\subsection{Type A}

Let $X=\{x_1,x_2,\ldots\}$ be a totally ordered set of commutative variables. If $\alpha=(\alpha_1,\ldots,\alpha_\ell)\models n$ then the \emph{monomial quasisymmetric function} $M_\alpha$ and the \emph{fundamental quasisymmetric function} $F_\alpha$ are defined as
\[ M_\alpha := \sum_{1\leq i_1<\cdots<i_\ell} x_{i_1}^{\alpha_1}\cdots x_{i_\ell}^{\alpha_\ell} \qand
F_\alpha := \sum_{ \substack{ 1\le i_1\le \cdots\le i_n\\ j\in D(\alpha) \Rightarrow i_j<i_{j+1} }} x_{i_1}\cdots x_{i_n}. \]
One see that $F_\alpha=\sum_{\alpha\cleq\beta}M_\beta$ is the generating function of all fillings of the ribbon $\alpha$ with positive integers such that each row is weakly increasing from left to right and each column is strictly decreasing from top to bottom. 

The Hopf algebra $\QSym$ has two bases $\{F_\alpha\}$ and $\{M_\alpha\}$ where $\alpha$ runs through all compositions. Given $w\in \SS_n$, let $F_w:=F_\alpha$ where $\alpha\models n$ satisfies $D(\alpha)=D(w)$. The product of $\QSym$ is determined by 
\begin{equation}
F_uF_v=\sum_{w\in u\shuffle v} F_w, \qquad \forall u\in\SS_m,\ \forall v\in\SS_n.
\end{equation}
Here $u\shuffle v$ is the set of all permutations in $\SS_{m+n}$ obtained by shuffling $u(1),\ldots,u(m)$ and $v(1)+m,\ldots,v(n)+m$. This is called the \emph{(shifted) shuffle product}. For example, ${\color{red}21}\shuffle{12} = \{{\color{red}21}34, {{\color{red}2}3{\color{red}1}4}, {3{\color{red}21}4}, {{\color{red}2}34{\color{red}1}}, {3{\color{red}2}4{\color{red}1}}, {34{\color{red}21}} \}.$

On the other hand, the coproduct of $\QSym$ is defined by $\Delta f(X):=f(X+Y)$ for all $f\in\QSym$, where $X+Y=\{x_1,x_2,\ldots,y_1,y_2,\ldots\}$ is a totally ordered set of commutative variables. If $\alpha=(\alpha_1,\ldots,\alpha_\ell)\models n$ then
\[ \Delta F_\alpha = \sum_{0\le i\le n} F_{\alpha_{\leq i}}\otimes F_{\alpha_{>i}} \qand
\Delta M_\alpha = \sum_{0\le i\le \ell} M_{(\alpha_1,\ldots,\alpha_i)}\otimes M_{(\alpha_{i+1},\ldots,\alpha_\ell)} \]
where $\alpha_{\le i}$ and $\alpha_{>i}$ are two ribbons obtained by splitting the ribbon $\alpha$ between its $i$th and $(i+1)$th boxes in the same order as the reading word of a tableau of shape $\alpha$. For instance, one has
$ \Delta F_{12} = 1\otimes F_{12}+ F_1 \otimes F_{2}+ F_{11} \otimes F_1 + F_{12} \otimes1. $

Let $\bX=\{\bx_i:i\in\ZZ\}$ be a totally ordered set of noncommutative variables. The Hopf algebra $\NSym$ is the free associative algebra $\ZZ\langle\bh_1,\bh_2,\ldots\rangle$ generated by $\bh_k:=\sum_{i_1\leq\cdots\leq i_k}\bx_{i_1}\cdots \bx_{i_k}$ for all $k\ge1$. It has two bases $\{\bh_\alpha\}$ and $\{\bs_\alpha\}$ where $\alpha$ runs through all compositions. If $\alpha=(\alpha_1,\ldots,\alpha_\ell)\models n$ then the \emph{complete homogeneous noncommutative symmetric function} $\bh_\alpha$ and the \emph{noncommutative ribbon Schur function} $\bs_\alpha$ are defined by\begin{equation}\label{eq:bhbs}
\bh_\alpha:=\bh_{\alpha_1}\cdots \bh_{\alpha_\ell} = \sum_{\beta\cleq \alpha}\bs_\beta \qand
\bs_\alpha:=\sum_{\beta\cleq\alpha}(-1)^{\ell(\alpha)-\ell(\beta)}\bh_\beta.
\end{equation}
It is more common to define $\NSym$ using noncommutative variables indexed by positive integers, but we need $\bX$ in order to define type B and D analogues of $\NSym$ later.
The graded Hopf algebras $\QSym$ and $\NSym$ are dual to each other, with bases $\{M_\alpha\}$ and $\{F_\alpha\}$ dual to $\{\bh_\alpha\}$ and $\{\bs_\alpha\}$, respectively.

We next extend the definition of $\bs_\alpha$ to generalized ribbon shapes. If $\alpha$ is a generalized ribbon of size $n$ then we define $\bs_\alpha$ to be the sum of $\bx_\tau$ for all semistandard tableaux of shape $\alpha$, where $\bx_\tau:=\bx_{w_1}\cdots\bx_{w_n}$ if $w(\tau)=w_1\cdots w_n$. We will show that this definition of $\bs_\alpha$ agrees with the earlier definition. 

Suppose that $\alpha = \alpha^1\oplus\cdots\oplus\alpha^k$ and $\beta = \beta^1\oplus\cdots\oplus\beta^\ell$ are generalized ribbons. We write
\begin{eqnarray*}
 \alpha \cdot \beta &:=& \alpha^1\oplus\cdots\oplus(\alpha^k\cdot\beta^1) \oplus\cdots\oplus\beta^\ell \qand \\
 \alpha \rhd \beta &:=& \alpha^1\oplus\cdots\oplus(\alpha^k\rhd\beta^1) \oplus\cdots\oplus\beta^\ell.
\end{eqnarray*}
Let $\tau$ and $\eta$ be semistandard tableaux of generalized ribbon shape $\alpha$ and $\beta$, respectively. There is a uniquely way to glue $\tau$ and $\eta$ to get a semistandard tableau $\tau*\eta$ of generalized ribbon shape either $\alpha\cdot\beta$ or $\alpha\rhd\beta$, depending on whether the last entry of $w(\tau)$ is strictly larger than the first entry of $w(\eta)$. The reading word of the semistandard tableau $\tau*\eta$ is equal to the concatenation of $w(\tau)$ and $w(\eta)$. Some examples are given below.
\[ \begin{matrix}
\raisebox{-5pt}{ \young(:22,14) }& * & \raisebox{-5pt}{ \young(:1334,14)} 
& = & \raisebox{-12pt}{ \young(:::1334,::14,:22,14) } \\ \\
\raisebox{-5pt}{ \young(:13,23) } & * & \raisebox{-8pt}{ \young(::44,:15,34) } 
& = & \raisebox{-12pt}{ \young(:::::44,::::15,:1334,23) }
\end{matrix} \]\vskip5pt

\begin{proposition}\label{prop:NSym}
Let $\alpha=\alpha^1\oplus\cdots\oplus\alpha^k$ be a generalized ribbon and let $\beta$ be a another generalized ribbon. Then
\[ \bs_{\alpha}\cdot\bs_\beta = \bs_{\alpha\cdot\beta} + \bs_{\alpha\rhd\beta} \qand
\bs_\alpha = \bs_{\alpha^1} \cdots \bs_{\alpha^k} = \sum_{\gamma\in[\alpha]} \bs_\gamma. \]
\end{proposition}

\begin{proof}
The above argument on $\tau*\eta$ implies the first equality. By induction on $k$ one has the second equality.
\end{proof}

\begin{remark}
For an arbitrary skew shape $\lambda/\mu$, one can still define $\bs_{\lambda/\mu}$ to be the sum of $\bx_\tau$ for all semistandard tableaux of shape $\bs_{\lambda/\mu}$. Using the $P$-partition theory one can show that $\bs_{\lambda/\mu}$ lies in the Hopf algebra of free quasisymmetric functions (cf. \cite{NCSF_VI}), but not in $\NSym$ in general.
\end{remark}

In particular, let $\alpha=(\alpha_1,\ldots,\alpha)$ be a composition. We define $\bh_\alpha:=\bs_{\alpha_1\oplus\cdots\oplus\alpha_\ell}$. One sees that $\bh_k=\sum_{i_1\leq\cdots\leq i_k}\bx_{i_1}\cdots \bx_{i_k}$ holds for all $k\ge1$ according to this definition. Moreover, Proposition~\ref{prop:NSym} implies \eqref{eq:bhbs}. Thus the two definitions of $\bh_\alpha$ and $\bs_\alpha$ agree with each other for all compositions $\alpha$.

Proposition~\ref{prop:NSym} immediately implies a (well-known) product formula for $\NSym$. On the other hand, the coproduct of $\NSym$ is defined by $ \Delta\bh_k=\sum_{0\le i\le k} \bh_i\otimes\bh_{k-i} $ where $\bh_0:=1$. If $\alpha=(\alpha_1,\ldots,\alpha_\ell)$ is a composition then one has $\Delta(\bh_\alpha)=\Delta(\bh_{\alpha_1})\cdots\Delta(\bh_{\alpha_\ell})$ since the coproduct of a Hopf algebra must be an algebra homomorphism. We next provide a more explicit coproduct formula for $\NSym$.

One can decompose a generalized ribbon $\alpha$ into a disjoint union of two generalized ribbons $\beta$ and $\gamma$ in the following way. First fill its boxes with two symbols $\beta$ and $\gamma$ such that each row is weakly increasing from left to right and each column is weakly increasing from top to bottom, where $\beta$ is viewed as less than $\gamma$. Then the boxes filled with $\beta$ form a generalized ribbon denoted by $\beta$, and the boxes filled with $\gamma$ form a generalized ribbon denoted by $\gamma$. We write $\alpha=\beta\sqcup \gamma$ for this decomposition. For example, if $\alpha=(1,3)$ then all decompositions $\alpha=\beta\sqcup\gamma$ are given below.
\vskip5pt\[ \young(\beta\beta\beta,\beta)\qquad
\young(\beta\beta\beta,\gamma)\qquad
\young(\beta\beta\gamma,\beta)\qquad
\young(\beta\beta\gamma,\gamma)\qquad
\young(\beta\gamma\gamma,\beta)\qquad
\young(\beta\gamma\gamma,\gamma)\qquad
\young(\gamma\gamma\gamma,\gamma)
\]\vskip5pt

\begin{theorem}\label{thm:CoprodSchur}
For any generalized ribbon $\alpha$ one has
\[ \Delta\bs_\alpha = \sum_{\alpha=\beta\sqcup\gamma} \bs_{\beta}\otimes\bs_{\gamma} = \sum_{\substack{\alpha=\beta\sqcup\gamma \\ \beta'\in\,[\beta] \\ \gamma'\in\,[\gamma]}} \bs_{\beta'}\otimes\bs_{\gamma'}\]
\end{theorem}

\begin{proof}
We only need to prove the first equality, which immediately implies the second equality by Proposition~\ref{prop:NSym}.

If $\alpha=\alpha^1\oplus\cdots\oplus\alpha^k$ is a generalized-ribbon with connected components $\alpha^i$, then one has $\Delta(\bs_\alpha)=\Delta(\bs_{\alpha^1})\cdots \Delta(\bs_{\alpha^k})$ by  Proposition~\ref{prop:NSym}. A decomposition $\alpha=\beta\sqcup\gamma$ is equivalent to decompositions $\alpha^i=\beta^i\sqcup\gamma^i$ for all $i\in[k]$, where $\beta^i$ and $\gamma^i$ are the $i$th components of $\beta$ and $\gamma$, respectively. Hence it suffices to prove the first equality assuming $\alpha$ is a composition. To this aim, we prove the following equality by induction on the length of $\alpha$:
\begin{equation}\label{eq:Deltah}
\Delta \bh_\alpha = \sum_{\beta\sqcup\gamma\cleq\alpha} \bs_\beta\otimes\bs_\gamma.
\end{equation}
It is trivial when $\ell(\alpha)=1$. If $\ell(\alpha)>1$ then one can write $\alpha = \alpha^1\cdot\alpha^2$ where $\alpha^1$ and $\alpha^2$ are compositions of smaller lengths. Since $\Delta$ is an algebra map, one has $\Delta \bh_{\alpha} =\Delta \bh_{\alpha^1}\cdot\Delta \bh_{\alpha^2}$. Using induction hypothesis one obtains
\[ \Delta \bh_{\alpha^1}\cdot\Delta \bh_{\alpha^2} = \sum_{ \substack{\beta^1\sqcup \gamma^1 \cleq \alpha^1 \\ \beta^2\sqcup\gamma^2 \cleq \alpha^2 }} \bs_{\beta^1}\bs_{\beta^2} \otimes \bs_{\gamma^1} \bs_{\gamma^2}. \]
We need to show that this is the same as \eqref{eq:Deltah}. Suppose that $\beta\sqcup \gamma \cleq \alpha$. We cut the ribbon $\beta\sqcup \gamma$ between its $n$th and $(n+1)$th boxes in the reading order, where $n=|\alpha^1|$. This gives two ribbons $\beta^1\sqcup\gamma^1 \cleq \alpha^1$ and $\beta^2\sqcup\gamma^2 \cleq\alpha^2$. We distinguish the following cases for the $n$th and $(n+1)$th boxes in $\beta\sqcup\gamma$.
\begin{equation}\label{eq:BetaGamma} 
\begin{ytableau} \beta^1 & \beta^2 \end{ytableau} \qquad
\raisebox{6pt}{\begin{ytableau}  \beta^2 \\ \beta^1 \end{ytableau}} \qquad\qquad
\begin{ytableau}  \gamma^1 & \gamma^2 \end{ytableau} \qquad
\raisebox{6pt}{\begin{ytableau}  \gamma^2 \\ \gamma^1 \end{ytableau}} \qquad\qquad
\begin{ytableau}  \beta^1 & \gamma^2 \end{ytableau} \qquad
\raisebox{6pt}{\begin{ytableau}  \beta^2 \\ \gamma^1 \end{ytableau}}
\end{equation}

In the first two cases one has $\gamma^1$ and $\gamma^2$ disconnected. 
Hence $\bs_\gamma = \bs_{\gamma^1}\bs_{\gamma^2}$. 
Also, either $\beta=\beta^1\cdot\beta^2$ or $\beta = \beta^1\rhd\beta^2$, and $\bs_{\beta^1\cdot\beta^2} + \bs_{\beta^1\rhd\beta^2} = \bs_{\beta^1}\bs_{\beta^2}$. 
Thus 
\[ \sum \bs_\beta\otimes\bs_\gamma =  \sum \bs_{\beta^1}\bs_{\beta^2} \otimes \bs_{\gamma^1}\bs_{\gamma^2} \]
where the first sum runs over all $\beta\sqcup\gamma\cleq\alpha$ belonging to the first two cases of \eqref{eq:BetaGamma}, and the second sum runs over all pairs $\beta^1\sqcup\gamma^1\cleq\alpha^1$ and $\beta^2\sqcup\gamma^2\cleq\alpha^2$ such that the last box of $\beta^1\sqcup\gamma^1$ is occupied by $\beta^1$ and the first box of $\beta^2\sqcup\gamma^2$ is occupied by $\beta^2$.

In the next two cases of \eqref{eq:BetaGamma} one has $\beta^1$ and $\beta^2$ disconnected, and either $\gamma=\gamma^1\cdot\gamma^2$ or $\gamma = \gamma^1\rhd\gamma^2$. Thus
\[
\sum \bs_\beta\otimes\bs_\gamma =  \sum \bs_{\beta^1}\bs_{\beta^2} \otimes \bs_{\gamma^1}\bs_{\gamma^2}
\]
where the first sum runs over all $\beta\sqcup\gamma\cleq\alpha$ belonging to the third and fourth cases of \eqref{eq:BetaGamma}, and the second sum runs over all pairs $\beta^1\sqcup\gamma^1\cleq\alpha^1$ and $\beta^2\sqcup\gamma^2\cleq\alpha^2$ such that the last box of $\beta^1\sqcup\gamma^1$ is occupied by $\gamma^1$ and the first box of $\beta^2\sqcup\gamma^2$ is occupied by $\gamma^2$.

Finally, in the last two cases of \eqref{eq:BetaGamma} one has $\beta^1$ and $\beta^2$  disconnected and $\gamma^1$ and $\gamma^2$ disconnected. Thus
\[
\sum \bs_\beta\otimes\bs_\gamma =  \sum \bs_{\beta^1}\bs_{\beta^2} \otimes \bs_{\gamma^1}\bs_{\gamma^2}
\]
where the first sum runs over all $\beta\sqcup\gamma\cleq\alpha$ belonging to the last two cases of \eqref{eq:BetaGamma}, and the second sum runs over all pairs $\beta^1\sqcup\gamma^1\cleq\alpha^1$ and $\beta^2\sqcup\gamma^2\cleq\alpha^2$ such that the last box of $\beta^1\sqcup\gamma^1$ is occupied by $\beta^1$ [$\gamma^1$ resp.] and the first box of $\beta^2\sqcup\gamma^2$ is occupied by $\gamma^2$ $[\beta^2$ resp.].

Thus \eqref{eq:Deltah} holds. Applying inclusion-exclusion completes the proof.
\end{proof}

\begin{remark}
Theorem~\ref{thm:CoprodSchur} suggests another way to define the coproduct of $\NSym$, that is, define $\Delta \mathbf{f} := \mathbf{f}(\mathbf X+\mathbf Y)$ where $\mathbf X=\{\bx_i:i\in\ZZ\}$ and $\mathbf Y=\{\mathbf y_j:j\in\ZZ\}$ are two sets of noncommutative variables such that $\bx_i<\mathbf y_j$ and $\bx_i\mathbf y_j=\mathbf y_j\bx_i$ for all $i,j$. 
The reader is referred to Grinberg and Reiner~\cite[\S8.1]{GrinbergReiner} for a detailed discussion of the similarly defined coproduct of the Hopf algebra of free quasisymmetric functions, which contains $\NSym$ as a Hopf subalgebra.
\end{remark}


Bergeron and Li~\cite{BergeronLi} showed that the Grothendieck groups $G_0(\H_\bullet(0))$ and $K_0(\H_\bullet(0))$ associated with the tower  $\H_\bullet(0): \H_0(0)\hookrightarrow \H_1(0) \hookrightarrow \H_2(0) \hookrightarrow\cdots$ of 0-Hecke algebras of type A are graded Hopf algebras whose  product and coproduct are defined by
\[ M \,\widehat{\otimes}\, N := (M\otimes N)\uparrow\,_{\H_m(0)\otimes\H_n(0)}^{\H_{m+n}(0)} \qand
\Delta M := \sum_{0\le i\le m} M\downarrow\,_{\H_i(0)\otimes\H_{m-i}(0)}^{\H_m(0)} \]
for all finitely generated (projective) $\H_m(0)$-modules $M$ and $\H_n(0)$-modules $N$. Following Krob and Thibon~\cite{KrobThibon}, we define the \emph{noncommutative characteristic} and the \emph{quasisymmetric characteristic} as 
\[ \begin{matrix}
\mathbf{ch}: & K_0(\H_\bullet(0)) & \to & \NSym & \text{and} & \mathrm{Ch}: & G_0(\H_\bullet(0)) & \to & \QSym \\
& \bP_\alpha & \mapsto & \bs_\alpha & &  & \bC_\alpha & \mapsto & F_\alpha  
\end{matrix} \]
where $\alpha$ runs through all compositions. For $G_0(\H_\bullet(0))$, a product formula is given in~\cite{NCSFIII,QSymNSymH0} and a coproduct formula is provided in~\cite{KrobThibon}, showing that the quasisymmetric characteristic Ch is a Hopf algebra isomorphism. For $K_0(\H_\bullet(0))$, a product formula is included in~Theorem~\ref{thm:IndP} and a coproduct formula is obtained below. Comparing these results with Proposition~\ref{prop:NSym} and Theorem~\ref{thm:CoprodSchur} one sees that the noncommutative characteristic $\mathbf{ch}$ is also a Hopf algebra isomorphism sending $\bP_\alpha$ to $\bs_\alpha$ for any generalized ribbon $\alpha$.

\begin{proposition}\label{prop:ResP}
If $\alpha$ is a generalized ribbon of size $m+n$ then
\[ \bP_\alpha\downarrow\,_{\H_{m,n}(0)}^{\H_{m+n}(0)}\ \cong 
\bigoplus_{\substack{\alpha=\beta\sqcup\gamma \\ |\beta|=m \\ |\gamma|=n} } \bP_\beta\otimes \bP_\gamma
\cong \bigoplus_{\substack{\alpha=\beta\sqcup\gamma \\ |\beta|=m,\ \beta'\in[\beta] \\ |\gamma|=n,\ \gamma'\in[\gamma]} } \bP_{\beta'}\otimes \bP_{\gamma'}.\]
\end{proposition}

\begin{proof}
Let $\tau$ be a standard tableau of shape $\alpha$. The entries $1,\ldots,m$ in $\tau$ form a standard tableau $\tau_{\le m}$ of shape $\beta$. Subtracting $m$ from the entries $m+1,\ldots,m+n$ in $\tau$ gives a standard tableau $\tau_{>m}$ of shape $\gamma$. One has $\alpha=\beta\sqcup\gamma$, $|\beta|=m$, and $|\gamma|=n$. Conversely, if a composition $\alpha$ can be written as $\alpha=\beta \sqcup\gamma$ for a pair of generalized ribbons $\beta$ and $\gamma$ of sizes $m$ and $n$, then every pair $(\tau,\tau')$ of standard tableaux of shapes $\beta$ and $\gamma$ can be glued together to form a standard tableau of shape $\alpha$, with a value of $m$ added to every entry in $\tau'$. Moreover, the $\H_{m,n}(0)$-action is preserved by this correspondence. Hence one has the first desired isomorphism of $\H_{m,n}(0)$-modules. The second desired isomorphism then follows from Theorem~\ref{thm:IndP}.
\end{proof}

\subsection{Type B}
Let $x_0,x_1,x_2,\ldots$ be commutative variables. Chow~\cite{Chow} introduced a type B analogue $\QSym^B$ for $\QSym$ with two bases consisting of $\MB_\alpha$ and $\FB_\alpha$, respectively, for all pseudo-compositions $\alpha$. If $\alpha=(\alpha_1,\ldots,\alpha_\ell)$ is a pseudo-composition of $n$ then the \emph{type B monomial quasisymmetric function} $M^B_\alpha$ and the \emph{type B fundamental quasisymmetric function} $\FB_\alpha$ are defined as
\[ \MB_\alpha := \sum_{0<i_2<\cdots<i_\ell} x_0^{\alpha_1}x_{i_2}^{\alpha_2}\cdots x_{i_\ell}^{\alpha_\ell} \qand 
\FB_\alpha := \sum_{ \substack{ 0\le i_1\le \cdots\le i_n\\ j\in D(\alpha) \Rightarrow i_j<i_{j+1} }} x_{i_1}\cdots x_{i_n}. \]
One sees that $\FB_\alpha=\sum_{\alpha\cleq\beta}\MB_\beta$ is the generating function of all fillings of the pseudo-ribbon $\alpha$ with nonnegative integers such that each row is weakly increasing from left to right and each column is strictly decreasing from top to bottom, \emph{including the extra 0-box}.  

Chow~\cite{Chow} showed that $\QSym^B$ is both an algebra and a coalgebra (but not a Hopf algebra), and is also a right graded $\QSym$-comodule defined by replacing the variables $x_0,x_1,x_2,\ldots$ with $x_0,x_1,x_2,\ldots,y_1,y_2,\ldots$. 
The $\QSym$-comodule structure on $\QSym^B$ is determined by the following formulas
\[ \FB_\alpha \mapsto \sum_{0\le i\le n} \FB_{\alpha_{\le i}}\otimes F_{\alpha_{>i}} \qand 
\MB_\alpha \mapsto \sum_{0\le i\le\ell} \MB_{(\alpha_1,\ldots,\alpha_i)}\otimes M_{(\alpha_{i+1},\ldots,\alpha_\ell)} \]
for all pseudo-compositions $\alpha=(\alpha_1,\ldots,\alpha_\ell)$ of $n$, where $\alpha_{\le i}$ and $\alpha_{>i}$ are obtained by cutting the pseudo-ribbon $\alpha$ between its $i$th and $(i+1)$th boxes in the reading order.

Chow~\cite{Chow} also introduced a right graded $\NSym$-module $\NSym^B$ dual to the $\QSym$-comodule $\QSym^B$. We next use tableaux of pseudo-ribbon shapes to study $\NSym^B$ and realize its elements as formal power series.

Let $\alpha$ be a generalized pseudo-ribbon. A \emph{(type B) semistandard tableau of shape $\alpha$} is a filling of $\alpha$ with integers such that every row is weakly increasing from left to right and every column is strictly increasing from top bottom, \emph{including the extra 0-box}. Reading these integers in $\tau$, \emph{excluding the extra $0$-box}, from left to right on each row and proceeding from the bottom row toward the top row, gives the \emph{reading word} $w(\tau)$ of $\tau$. Some examples are given below.
\vskip5pt
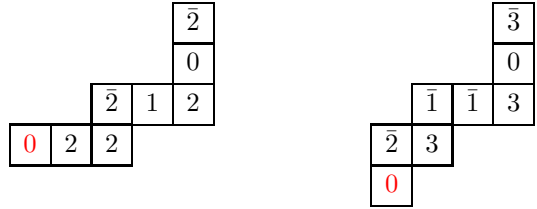
\[ \begin{ytableau} 
\none & \none & \none & \none & \bar2 \\
\none & \none & \none & \none & 0 \\
 \none & \none & \bar 2 & 1 & 2 \\
 {\color{red}0} & 2 & 2 
\end{ytableau}
 \qquad \qquad \qquad
\begin{ytableau} 
\none & \none & \none & \bar3 \\
\none & \none & \none & 0 \\
\none & \bar1 & \bar1 & 3 \\
 \bar2 & 3 \\
{\color{red}0} 
\end{ytableau}  \]\vskip5pt

Let $\bX:=\{\bx_i:i\in\mathbb Z\}$ be a set of noncommutative variables. Given a generalized pseudo-ribbon $\alpha$ we define $\bsB_\alpha$ to be the sum of $\bx_\tau$ for all semistandard tableaux $\tau$ of shape $\alpha$, where $\bx_\tau:=\bx_{w_1}\cdots\bx_{w_n}$ if $w(\tau)=w_1\cdots w_n$.

Let $\alpha=\alpha^1\oplus\cdots\oplus\alpha^k$ be a generalized pseudo-ribbon.
Let $\beta=\beta^1\oplus\cdots\beta^\ell$ be a generalized ribbon. Write
\[ \alpha \cdot \beta = \alpha^1\oplus\cdots\oplus(\alpha^k\cdot\beta^1) \oplus\cdots\oplus\beta^\ell \qand
\alpha \rhd \beta = \alpha^1\oplus\cdots\oplus(\alpha^k\rhd\beta^1) \oplus\cdots\oplus\beta^\ell. \]
Suppose that $\tau$ and $\eta$ are semistandard tableaux of shape $\alpha$ and $\beta$, respectively. There is a uniquely way to glue $\tau$ and $\eta$ to get a type B semistandard tableau $\tau*\eta$ of shape either $\alpha\cdot\beta$ or $\alpha\rhd\beta$, depending on whether the last entry of $w(\tau)$ is strictly larger than the first entry of $w(\eta)$. One sees that $w(\tau*\eta)$ equals the concatenation of $w(\tau)$ and $w(\eta)$.

\begin{proposition}\label{prop:NSymB}
Let $\alpha=\alpha^1\oplus\cdots\oplus\alpha^k$ be a generalized pseudo-ribbon.
Let $\beta$ be a generalized ribbon. Then
\[ \bsB_{\alpha}\cdot\bs_\beta = \bsB_{\alpha\cdot\beta} + \bsB_{\alpha\rhd\beta} \qand 
\bsB_\alpha = \bsB_{\alpha^1} \cdot \bs_{\alpha^2} \cdots \bs_{\alpha^k} = \sum_{\gamma\in[\alpha]} \bsB_\gamma.\]
\end{proposition}

\begin{proof}
The above argument on $\tau*\eta$ shows the first equality. By induction on $k$ one proves the second equality.
\end{proof}

We define $\NSym^B:=\bigoplus_{n\ge0}\NSym^B_n$ where $\NSym^B_n$ is the $\ZZ$-span of the $\bsB_\alpha$ for all $\alpha\modelsB n$. 
For each $a\in\ZZ^n$ there exists a unique semistandard tableau of pseudo-ribbon shape whose reading word is $a$. This implies that $\{\bsB_\alpha:\alpha\modelsB n, n\ge0\}$ is linearly independent and hence a basis for $\NSym^B$. 

If $\alpha=(\alpha_1,\ldots,\alpha_\ell)$ is a pseudo-composition then define $\bhB_\alpha:=\bsB_{\alpha_1\oplus\cdots\oplus\alpha_\ell}$. Proposition~\ref{prop:NSymB} implies that 
\[ \bhB_\alpha=\bhB_{\alpha_1}\cdot \bh_{\alpha_2}\cdots\bh_{\alpha_\ell} = \sum_{\beta\cleq \alpha}\bsB_\beta\]
Hence $\NSym^B$ admits another basis $\{\bhB_\alpha:\alpha\modelsB n, n\ge0\}$ and becomes a free right $\NSym$-module with a basis $\{\bhB_{k}:k\ge0\}$. This module structure was previously studied  by Chow~\cite{Chow} but not realized as formal power series. Note that if $\alpha$ is a composition then $\bh_\alpha=\bhB_{0\cdot\alpha}$ and $\bs_\alpha=\bsB_{\alpha}+\bsB_{0\cdot\alpha}$ (cf. Chow~\cite{Chow}).

Let $\HB_\bullet(0): \HB_0(0)\hookrightarrow \HB_1(0) \hookrightarrow \HB_2(0) \hookrightarrow\cdots$ be the tower of 0-Hecke algebras of type B. We define a \emph{type B noncommutative characteristic}
\[ \begin{matrix}
\mathrm{Ch}^B: & G_0(\HB_\bullet(0)) & \xrightarrow\sim & \QSym^B \\
& \bCB_\alpha & \mapsto & \FB_\alpha  
\end{matrix}\]
and a \emph{type B quasisymmetric characteristic} 
\[\begin{matrix}
 \mathbf{ch}^B: & K_0(\HB_\bullet(0)) & \xrightarrow\sim & \NSym^B \\
&\bPB_\alpha & \mapsto & \bsB_\alpha
\end{matrix} \]
where $\alpha$ runs through all pseudo-compositions. Since there is an embedding $\H^B_m(0)\otimes\H_n(0) \cong \HB_{m,n}(0)\subseteq \HB_{m+n}(0)$, we define a right action of $K_0(\H_\bullet(0))$ on $K_0(\HB_\bullet(0))$ and  a right coaction of $G_0(\H_\bullet(0))$ on $G_0(\HB_\bullet(0))$ by
\[ M \,\widehat{\otimes}\, N := (M\otimes N)\uparrow\,_{\HB_{m,n}(0)}^{\HB_{m+n}(0)} \qand
\Delta Q := \sum_{0\le i\le n} Q\downarrow\,_{\HB_{i,n-i}(0)}^{\HB_{n}(0)} \]
for all $M\in K_0(\HB_m(0))$, $N\in K_0(\H_n(0))$, and $Q\in G_0(\HB_n(0))$. 

\begin{proposition}
(i) $K_0(\HB_\bullet(0))$ is a right graded $K_0(\H_\bullet(0))$-module isomorphic to the right graded $\NSym$-module $\NSym^B$ via the noncommutative characteristic maps $\mathbf{ch}^B$ and $\mathbf{ch}$.

\noindent (ii) $G_0(\HB_\bullet(0))$ is a right graded $G_0(\H_\bullet(0))$-comodule isomorphic to the right graded $\QSym$-comodule $\QSym^B$ via the quasisymmetric characteristic maps $\mathrm{Ch}^B$ and $\mathrm{Ch}$.

\noindent(iii) The graded module and comodule structures in (i) and (ii) are dual.
\end{proposition}

\begin{proof}
Theorem~\ref{thm:IndPB} and Proposition~\ref{prop:NSymB} imply (i). The Frobenius reciprocity implies (ii) and (iii).
\end{proof}

\begin{remark}
One can define an action of $G_0(\H_\bullet(0))$ on $G_0(\HB_\bullet(0))$ and a coaction of $K_0(\H_\bullet(0))$ on $K_0(\HB_\bullet(0))$ in a similar way. It is relatively easy to obtain a formula for the former, but not for the latter---for example, it is not easy to extend Proposition~\ref{prop:ResP} to type B due to the extra 0-box in standard tableaux of pseudo-ribbion shapes. Our tableau approach does not suggest a natural way to define a $\QSym$-action on $\QSym^B$ and a $\NSym$-coaction on $\NSym^B$ such that the characteristic maps $\mathbf{ch}^B$ and $\mathrm{Ch}^B$ preserve these module and comodule structures. Moreover, it is not clear whether these module and comodule structures are dual to each other, as the Frobenius reciprocity does not apply to them. We will solve these problems using a different approach in another paper~\cite{ChH0}.
\end{remark}

\subsection{Type D}
Our results in type D are similar to type B. Let $X=\{x_i:i\in\mathbb Z\}$ be a set of commutative variables. For every pseudo-composition $\alpha$ of $n\ge2$, define 
\[ \MD_\alpha := \sum_{\substack{ i_0\leq i_1\leq\cdots\leq i_n \\ j\in D(\alpha)\Leftrightarrow i_j<i_{j+1} }} x_{i_1} \cdots x_{i_n} \qand
\FD_\alpha := \sum_{ \substack{ i_0\leq i_1\leq\cdots\leq i_n \\ j\in D(\alpha)\Rightarrow i_j<i_{j+1} }} x_{i_1} \cdots x_{i_n}. \]
Here $i_0:=-i_2$. 
One sees that $\FD_\alpha=\sum_{\alpha\cleq\beta}\MD_\beta$ is the generating function for all fillings of the pseudo-ribbon $\alpha$ with integers such that each row is weakly increasing from left to right and each column is strictly decreasing from top to bottom, \emph{with the extra $0$ interpreted as $-w(2)$}, where $w$ is the reading word of this filling. For each $a\in \ZZ^n$ there exists at most one such filling of the pseudo-ribbon $\alpha$ such that the reading word is a rearrangement of $a$. Hence $\{\FD_\alpha:\alpha\modelsD n,\ n\ge2\}$ is linearly independent.

We define $\QSym^D_n$ to be the $\ZZ$-span of $\{\FD_\alpha:\alpha\modelsD n\}$ and define $\QSym^D:=\bigoplus_{n\geq2}\QSym^D_n$. Then $\{\FD_\alpha:\alpha\modelsD n,\ n\ge2\}$ is a basis for $\QSym^D$. Another basis for $\QSym^D$ consists of $\MD_\alpha$ for all $\alpha\modelsD n\ge2$.

Let $X+Y=\{\ldots,x_{-2},x_{-1},x_0,x_1,x_2,\ldots,y_1,y_2,\ldots\}$ be a totally ordered set of commutative variables. One has a canonical projection
\[ \FF[X+Y]\cong\FF[X]\otimes\FF[Y]\twoheadrightarrow \FF[X]_{\ge2}\otimes\FF[Y]\]
where $\FF[X]_{\ge2}$ is the span of polynomials in X with degree at least 2. This includes a right $\QSym$-coaction on $\QSym^D$:
\[ \QSym^D(X)\mapsto \QSym^D(X+Y)\twoheadrightarrow \QSym^D(X)\otimes\QSym(Y). \]
If $\alpha=(\alpha_1,\ldots,\alpha_\ell)$ is a pseudo-composition of $n\ge2$ and $k$ is the smallest positive integer such that $\alpha_1+\cdots+\alpha_k\ge2$ then this coaction satisfies
\[ \FD_\alpha \mapsto \sum_{2\le i\le n} \FD_{\alpha{\le i}}\otimes F_{\alpha_{>i}} \qand 
\MD_\alpha \mapsto \sum_{k\le i\le\ell} \MD_{(\alpha_1,\ldots,\alpha_i)}\otimes M_{(\alpha_{i+1},\ldots,\alpha_\ell)}.\]

We next define a type D analogue of $\NSym$. Let $\alpha$ be a generalized pseudo-ribbon of size $n\ge2$. A \emph{type D semistandard tableau} $\tau$ of shape $\alpha$ is a filling of $\alpha$ with integers such that each row is weakly increasing from left to right and each column is strictly increasing from top to bottom, \emph{with the extra 0 interpreted as $-w(2)$}, where $w=w(\tau)$ is the reading word of $\tau$ obtained similarly as in type B. For example, two type D semistandard tableaux are illustrated below, whose reading words are $\bar22\bar21	202$ and $\bar21\bar1\bar10\bar2\bar3$, respectively.
\vskip5pt\[ \begin{ytableau} 
\none & \none & \none & \none & \bar2 \\
\none & \none & \none & \none & 0 \\
 \none & \none & \bar 2 & 1 & 2 \\
 {\color{red}0} & \bar 2 & 2 
\end{ytableau}
 \qquad \qquad \qquad
\begin{ytableau} 
\none & \none & \none & \bar3 \\
\none & \none & \none & \bar2 \\
\none & \bar1 & \bar1 & 0 \\
 \bar2 & 1 \\
{\color{red}0} 
\end{ytableau}  \]\vskip5pt
We define $\bsD_\alpha$ to be the sum of $\bx_\tau$ for all type D semistandard tableaux $\tau$ of shape $\alpha$, where $\bx_\tau:=\bx_{w_1}\cdots\bx_{w_n}$ if the reading word of $\tau$ is $w(\tau)=w_1\cdots w_n$.

\begin{proposition}\label{prop:NSymD}
Let $\alpha$ be a generalized pseudo-ribbon of size $n\ge2$. Then
\[ \bsD_\alpha = \sum_{\gamma\in[\alpha]} \bsD_\gamma.\]
If in addition $\alpha=\alpha^1\oplus\cdots\oplus\alpha^k$ and $|\alpha^1|\ge2$ then 
\[\bsD_\alpha = \bsD_{\alpha^1} \cdot \bs_{\alpha^2} \cdots \bs_{\alpha^k}.\]
If $\beta$ is a generalized ribbon then
\[ \bsD_{\alpha}\cdot\bs_\beta = \bsD_{\alpha\cdot\beta} + \bsD_{\alpha\rhd\beta}.\]
\end{proposition}

\begin{proof}
If $\tau$ is a type D semistandard tableau of shape $\alpha$ with connected components $\tau^1,\ldots,\tau^k$, then $\tau^1*\cdots*\tau^k$ is a type D semistandard tableau of shape $\gamma\in[\alpha]$. As this procedure can be reversed, the first desired equality holds. If in addition $\alpha=\alpha^1\oplus\cdots\oplus\alpha^k$ and $|\alpha^1|\ge2$, then $\tau$ is type D semistandard if and only if $\tau^1$ is type D semistandard and $\tau^2,\ldots,\tau^k$ are semistandard. This shows the second desired equality. Finally, if $\eta$ is a semistandard tableau of a generalized ribbon shape $\beta$, then $\tau*\eta$ is a type D semistandard tableau of shape $\alpha\cdot\beta$ or $\alpha\rhd\beta$. This implies the third desired equality.
\end{proof}

We define $\NSym^D:=\bigoplus_{n\ge2}\NSym^D_n$ where $\NSym^D_n$ is the $\ZZ$-span of the $\bsD_\alpha$ for all $\alpha\modelsD n$. For each $a\in\ZZ^n$ there exists a unique type D semistandard tableau whose reading word is $a$. This implies that the spanning set $\{\bsD_\alpha:\alpha\modelsD n, n\ge2\}$ linearly independent and hence a basis for $\NSym^D$. 

If $\alpha=(\alpha_1,\ldots,\alpha_\ell)$ is a pseudo-ribbon then we define $\bhD_\alpha:=\bsD_{\alpha_1\oplus\cdots\oplus\alpha_\ell}$. Proposition~\ref{prop:NSymD} implies that  
\[ \bhD_\alpha = \sum_{\beta\cleq \alpha}\bsD_\beta.\]
Hence $\NSym^D$ admits another basis $\{\bhD_\alpha:\alpha\modelsD n, n\ge2\}$.

The last equality in Proposition~\ref{prop:NSymD} shows that $\NSym^D$ is a graded right $\NSym$-module. One also has $\bhD_\alpha\cdot\bh_\beta=\bhD_{\alpha\cdot \beta}$ if $\alpha\modelsD m\ge2$ and $\beta\models n\ge0$. Hence the $\NSym$-module $\NSym^D$ is dual to the $\QSym$-comodule $\QSym^D$ , with bases $\{\bsD_\alpha\}$ and $\{\bhD_\alpha\}$ dual to $\{\FD_\alpha\}$ and $\{\MD_\alpha\}$, respectively.\vskip3pt

Let $\HD_\bullet(0): \HD_2(0)\hookrightarrow \HD_3(0) \hookrightarrow\cdots$ be the tower of 0-Hecke algebras of type D. We define a \emph{type D noncommutative characteristic}
\[ \begin{matrix}
\mathrm{Ch}^D: & G_0(\HD_\bullet(0)) & \xrightarrow\sim & \QSym^D  \\
& \bCD_\alpha & \mapsto & \FD_\alpha  
\end{matrix} \]
and a \emph{type D quasisymmetric characteristic}
\[ \begin{matrix}
 \mathbf{ch}^B: & K_0(\HD_\bullet(0)) & \xrightarrow\sim & \NSym^D \\
 & \bPD_\alpha & \mapsto & \bsD_\alpha
 \end{matrix} \] 
where $\alpha$ runs through all pseudo-compositions of $n\ge2$. 

We also define a right action of $K_0(\H_\bullet(0))$ on $K_0(\HD_\bullet(0))$ and a right coaction of $G_0(\H_\bullet(0))$ on $G_0(\HD_\bullet(0))$ by
\[ M \,\widehat{\otimes}\, N := (M\otimes N)\uparrow\,_{\HD_{m,n}(0)}^{\HD_{m+n}(0)} \qand 
\Delta Q := \sum_{2\le i\le n} Q\downarrow\,_{\HD_{i,n-i}(0)}^{\HD_{n}(0)} \]
for all $M\in K_0(\HD_m(0))$ ($m\ge2$), $N\in K_0(\H_n(0))$, and $Q\in G_0(\HD_n(0))$. 

\begin{proposition}
(i) $K_0(\HD_\bullet(0))$ is a right graded $K_0(\H_\bullet(0))$-module isomorphic to the right graded $\NSym$-module $\NSym^D$ via the noncommutative characteristic maps $\mathbf{ch}^D$ and $\mathbf{ch}$.

\noindent (ii) $G_0(\HD_\bullet(0))$ is a right graded $G_0(\H_\bullet(0))$-comodule isomorphic to the right graded $\QSym$-comodule $\QSym^D$ via the quasisymmetric characteristic maps $\mathrm{Ch}^D$ and $\mathrm{Ch}$.

\noindent(iii) The graded module and comodule structures in (i) and (ii) are dual.
\end{proposition}

\begin{proof}
Theorem~\ref{thm:IndPD} and Proposition~\ref{prop:NSymD} imply (i). The Frobenius reciprocity implies (ii) and (iii).
\end{proof}

\section{Other applications}\label{sec:other}

In this section we provide a few more applications of our tableau approach.

\subsection{Antipodes}\label{sec:antipodes}

The antipode of the Hopf algebra $\Sym \cong G_0(\CC\SS_\bullet)$ is defined by sending a skew Schur function $s_{\lambda/\mu}$ to $(-1)^n s_{\lambda^t/\mu^t}$ where $n=|\lambda/\mu|$. This antipode, up to a sign, corresponds to tensoring a simple $\CC\SS_n$-representation with the sign representation of $\CC\SS_n$ in the Grothendieck group $G_0(\CC\SS_\bullet)$. See, for example, Grinberg and Reiner~\cite[\S2.4, \S4.4]{GrinbergReiner}.

On the other hand, the antipodes of the two Hopf algebras $\QSym\cong G_0(\H_\bullet(0))$ and $\NSym \cong K_0(\H_\bullet(0))$ are given by $S(F_\alpha)=(-1)^{|\alpha|}F_{\alpha^t}$ and $S(\bs_\alpha)=(-1)^{|\alpha|}\bs_{\alpha^t}$ for all compositions $\alpha$~\cite[(5.9), (5.24)]{GrinbergReiner}. This can be interpreted by the representation theory of 0-Hecke algebras of type A: for any composition $\alpha$ one has $\mathbf{ch}(\bP_\alpha)=\bs_\alpha$, and  by Remark~\ref{rem:symmetryA} one obtains $\bP_{\alpha^t}$ from $\bP_{\alpha}$ by taking transpose of standard tableaux of shape $\alpha$. We also discussed certain symmetry between $\bPB_\alpha$ [$\bPD_\alpha$ resp.] and $\bPB_{\alpha^c}$ and [$\bPD_{\alpha^c}$ resp.] in Section~\ref{sec:H0B} [Section~\ref{sec:H0D} resp.]. Now we provide a uniform treatment.

By Fayers~\cite{Fayers}, there are two automorphisms $\H_W(0)$ of the $0$-Hecke algebra $H_W(0)$ of a finite Coxeter system $(W,S)$, which are defined as 
\[ \theta:\pi_s\mapsto -\pib_s=1-\pi_s \qand \phi:\pi_s\mapsto \pi_{w_0sw_0}, \quad \forall s\in S\]
where $w_0$ is the longest element of $W$. One sees that they are both involutions and commute with each other.  By Fayers~\cite{Fayers}, one has $\phi(\pi_s)=\pi_{\sigma(s)}$ for all $s\in S$, where $\sigma$ is an automorphism of the Coxeter diagram of $(W,S)$. 

Given an $\H_W(0)$-module $M$, letting $\pi_s$ act on $M$ by $\theta(\pi_s)$ [$\phi(\pi_s)$ resp.] for all $s\in S$ gives an $\H_W(0)$-module $\theta[M]$ [$\phi[M]$ resp.]. If $M$ happens to be a submodule of $\H_W(0)$ then one can directly apply $\theta$ [$\phi$ resp.] to it and get an $\H_W(0)$-module $\theta(M)$ and [$\phi(M)$ resp.]. Fayers~\cite{Fayers} showed that $\theta[\C_I^S] \cong \C_{I^c}$ and $\phi[\C_I^S] \cong \C_{\sigma(I)}$ for all $I\subseteq S$. We extend this result below.

\begin{proposition}\label{prop:AntiW}
Let $I\subseteq S$. Then one has isomorphisms of $\H_W(0)$-modules $\theta(\P_I^S) \cong \P_{I^c}^S \cong \theta[\P_I^S]$ and $\phi(\P_I^S) = \P_{\sigma(I)}^S \cong \phi[\P_I^S]$.
\end{proposition}

\begin{proof}
Norton~\cite{Norton} provided two decompositions
\[
\H_W(0) = \bigoplus_{I\subseteq S}  \H_W(0) \pib_{w_0(I)}\pi_{w_0(I^c)} = \bigoplus_{I\subseteq S} \H_W(0)\pi_{w_0(I)}\pib_{w_0(I^c)}.
\]
The first decomposition is discussed in Section~\ref{sec:HW0} whose summands are denoted by $\P_I^S$ for all $I\subseteq S$. Applying the automorphism $\theta$ to the first decomposition gives the second one. Thus $\theta(\P_I^S) = \H_W(0)\pi_{w_0(I)}\pib_{w_0(I^c)}$ is projective indecomposable with a basis $\{\pi_w\pib_{w_0(I^c)}: w\in W,\ D(w)=I\}$. If $s\in S$ and $w\in W$ with  $D(w)=I$ then 
\[
\pi_s \pi_{w}\pib_{w_0(I^c)}=
\begin{cases}
\pi_{w}\pib_{w_0(I^c)}, & {\rm if}\ i\in D(w^{-1}), \\
0, & {\rm if}\ i\notin D(w^{-1}),\ D(sw)\ne I, \\
\pi_{sw}\pib_{w_0(I^c)}, & {\rm if}\  i\notin D(w^{-1}),\ D(sw)=I.
\end{cases}
\]
Using $D(w_0(I)^{-1})=I$ one sees that the top of $\theta(\P_I^S)$ is isomorphic to $\C_{I^c}$. Thus $\theta(\P_I^S)\cong \P_{I^c}^S$. On the other hand, one can check that $\phi(\P_I^S) = \P_{\sigma(I)}^S$.

Next, one sees that $\theta[\H_W(0)]$ is still isomorphic to the regular representation of $\H_W(0)$ and decomposes as the direct sum of $\theta[\P_I^S]$ for all $I\subseteq S$. Each direct summand $\theta[\P_I^S]$ has a one-dimensional quotient isomorphic to $\C_{I^c}^S$. Hence $\theta[\P_I^S]\cong \P_{I^c}^S$. A similar argument shows $\phi[\P_I^S] \cong \P_{\sigma(I)}^S$.
\end{proof}

In particular, let $W=\SS_n$. One has $\phi(\pib_i)=\pib_{n-i}$ coming from the nontrivial automorphism of the Coxeter diagram of type $A_{n-1}$~\cite[Proposition 2.4]{Fayers}. 
If we write $\P_\alpha:=\P_I$ and $\C_\alpha:=\C_I$ where $I:=\{s_i: i\in D(\alpha)\}$ for all compositions $\alpha$ of $n$, then combining work of Fayers~\cite{Fayers}, Proposition~\ref{prop:AntiW}, and the observation $D(\rev(\alpha)) = \{ n-i: i\in D(\alpha)\}$, one has
\[ 
\phi\circ\theta[\C_\alpha]\cong \C_{\alpha^t} \qand 
\phi\circ \theta(\P_\alpha) \cong \phi\circ\theta[\P_\alpha] \cong \P_{\alpha^t}.
\]
This interprets the antipodes for $K_0(\H_\bullet(0))$ and $G_0(\H_\bullet(0))$ up to a sign. 

\begin{remark}
Let $\alpha\models n$. Recall that $\H_n(0)$ acts on $\bP_\alpha$ by \eqref{eq:RibbonAction}.   There is another $\H_n(0)$-action on $\bP_\alpha$ by 
\[ \pi_i(\tau): = \begin{cases}
\tau, &  \textrm{if $i$ is in a higher row of $\tau$ than $i+1$},\\
0, & \textrm{if $i$ is in the same row of $\tau$ as $i+1$}, \\
s_i(\tau), & \textrm{if $i$ is in a lower row of $\tau$ than $i+1$}.
\end{cases} \]
for all $i\in[n-1]$ and all standard tableaux $\tau$ of shape $\alpha$. This agrees with $\theta[\bP_\alpha]$ up to a sign and can be extended to generalized ribbons as well.
\end{remark}

\subsection{Polynomial Representations}
It is well known that the representation theory of symmetric groups can be described using polynomials instead of tableaux. We provide an analogue of this for the representation theory of 0-Hecke algebras of type A.

The symmetric group $\mathfrak S_n$ acts on the polynomial ring $\FF[x_1,\ldots,x_n]$ over an arbitrary field $\FF$ by permuting the variables $x_1,\ldots,x_n$. Given a partition $\lambda=(\lambda_1,\ldots,\lambda_\ell)$ of $n$, define 
\[ \Delta_\lambda:=\Delta[1,\lambda_1] \Delta[\lambda_1+1,\lambda_1+\lambda_2] \cdots \Delta[\lambda_1+\cdots+\lambda_{\ell-1}+1,\lambda_1+\cdots+\lambda_\ell] \]
where $\Delta[a,b]:=\prod_{a\le i<j\le b} (x_j-x_i)$. Then $\CC\SS_n\cdot \Delta_\lambda$ is isomorphic to the Specht module $S^{\lambda^t}$. See Lascoux~\cite{PolySn}.

On the other hand, the 0-Hecke algebra $\H_n(0)$ acts on the polynomial ring $\FF[x_1,\ldots,x_n]$ by the \emph{Demazure operators} $\pi_1,\ldots,\pi_{n-1}$, where
\begin{equation}\label{eq:Demazure}
\pi_if:= \frac{x_if-x_{i+1}s_i(f)}{x_i-x_{i+1}},\quad \forall f\in\FF[\bx].
\end{equation}

Let $\alpha=(\alpha_1,\ldots,\alpha_\ell)\models n$ and define $x_\alpha:=\prod_{i\in D(\alpha)} x_1\cdots x_i$. One has
\[ wx_\alpha = \prod_{i\in D(\alpha)} x_{w(1)}\cdots x_{w(i)}, \quad \forall w\in\SS_n. \]
A monomial in $\FF[x_1,\ldots,x_n]$ can be written as $x^d:=x_1^{d_1}\cdots x_n^{d_n}$ where $d=(d_1,\ldots,d_n)$ is a sequence of nonnegative integers. Let $\lambda(d)$ be the partition obtained from $d$ by rearranging its nonzero parts. Given two monomials $x^d$ and $x^e$, write $x^d\prec x^e$ if $\lambda(d)<_L\lambda(e)$, where ``$<_L$'' is the lexicographic order.

\begin{lemma}[Huang~\cite{H0CF}]\label{DemazureAtoms}
If $\alpha$ is a composition of $n$ and $w$ is a permutation in $\SS_n$ with $D(w)\subseteq D(\alpha)$, then
\[
\overline\pi_w x_{\alpha} =w x_{\alpha} +\sum_{x^d\prec x_{\alpha}}c_dx^d, \quad c_d\in\ZZ.
\]
\end{lemma}

Now we describe the representations of $\H_n(0)$ constructed in Section~\ref{sec:H0A} using polynomials instead of tableaux.

\begin{proposition}\label{prop:PolyM}
If $\alpha$ is a composition of $n$ then the $\H_n(0)$-module $\H_n(0) x_\alpha$ has a basis $\{ \pib_w x_\alpha: D(w)\subseteq D(\alpha)\}$ and is isomorphic to $\bM_\alpha$.
\end{proposition}

\begin{proof}
If $i\notin D(\alpha)$ then $\pib_i x_\alpha=0$ since $x_i$ and $x_{i+1}$ have the same exponent in $x_\alpha$. Hence $\H_n(0) x_\alpha$ is spanned by $\{\pib_w x_\alpha: D(w)\subseteq D(\alpha)\}$, which is triangularly related to the linearly independent set $\{wx_\alpha: w\in\SS_n,\ D(w)\subseteq D(\alpha)\}$ by Lemma~\ref{DemazureAtoms}. Hence $\{\pib_w x_\alpha: D(w)\subseteq D(\alpha)\}$ is a basis for $\H_n(0)x_\alpha$.

For any $w\in \SS_n$ with $D(w)\subseteq D(\alpha)$, there is a unique standard tableau of shape $\alpha_1\oplus\cdots\oplus\alpha_\ell$ whose reading word is $w$. This gives a vector space isomorphism $\H_n(0)x_\alpha \cong \bM_\alpha$. One can check that it preserves the $\H_n(0)$-actions, and thus is an isomorphism of $\H_n(0)$-modules.
\end{proof}

\begin{corollary}
Let $\alpha=\alpha^1\oplus\cdots\oplus\alpha^k$ be a generalized ribbon with connected components $\alpha_i\models n_i$ for $i=1,\ldots,k$, and let $|\alpha| = n$. Then the $\H_n(0)$-module $\H_n(0)\cdot \pib_{w_0(\alpha^1\rhd\cdots\rhd\alpha^k)} \cdot x_{\alpha^1\cdots\alpha^k}$ is isomorphic to $\bP_\alpha$ and has a basis
\[
\left\{ \pib_w x_{\alpha^1\cdots\alpha^k} : w\in\SS_n,\ D(\alpha^1\rhd\cdots\rhd\alpha^k) \subseteq D(w) \subseteq D(\alpha^1\cdots\alpha^k) \right\}.
\] 
\end{corollary}

\begin{proof}
By Proposition~\ref{prop:PolyM}, $\H_n(0)\cdot \pib_{w_0(\alpha^1\rhd\cdots\rhd\alpha^k)} \cdot x_{\alpha^1\cdots\alpha^k} \subseteq \H_n(0)\cdot x_{\alpha^1\cdots\alpha^k}$ has the desired basis. An element $\pib_w  x_{\alpha^1\cdots\alpha^k}$ in this basis corresponds to a standard tableau of shape $\alpha$ whose reading word is $w$. This gives the desired isomorphism between $\H_n(0)\cdot \pib_{w_0(\alpha^1\rhd\cdots\rhd\alpha^k)} \cdot x_{\alpha^1\cdots\alpha^k}$ and $\bP_\alpha$. 
\end{proof}

\begin{example}\label{ex:M211}
The $\H_n(0)$-module generated by $x_{211} = x_1^2x_2^2x_3$ is isomorphic to $\bM_{211}=\bP_{2\oplus1\oplus1}$ as illustrated below. The submodules $\H_n(0) \pib_2 x_{211}$, $\H_n(0) \pib_3 x_{211}$,  and $\H_n(0) \pib_2\pib_3\pib_2 x_{211}$ of $\H_n(0) x_{211}$ are isomorphic to $\bP_{21\oplus 1}$, $\bP_{2\oplus11}$, and $\bP_{211}$, respectively.
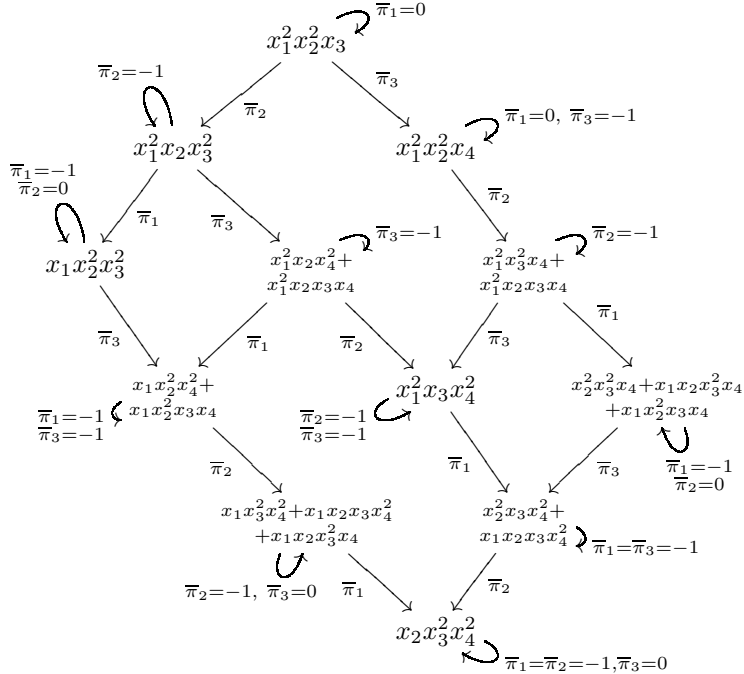
\[ \xymatrix @C=-5pt{ 
&&x_1^2x_2^2x_3\ar@(ru,r)[]^{\pib_1=0} \ar@{->}[ld]^{\pib_2} \ar[rd]^{\pib_3} \\
& x_1^2x_2x_3^2 \ar@(u,lu)[]_{\pib_2=-1} \ar[ld]^{\pib_1} \ar[rd]_{\pib_3}  & & x_1^2x_2^2x_4 \ar@(ru,r)[]^{\pib_1=0,\ \pib_3=-1} \ar[rd]^{\pib_2}  \\
x_1x_2^2x_3^2 \ar@(u,lu)[]_{\txt{$\substack{\pib_1=-1 \\ \pib_2=0}$}} \ar[rd]_{\pib_3} & & \txt{ $\substack{x_1^2x_2x_4^2+\\x_1^2x_2x_3x_4}$} \ar@(ru,r)[]^{\pib_3=-1} \ar[ld]^{\pib_1}  \ar[rd]_{\pib_2} & & \txt{$\substack{x_1^2x_3^2x_4+\\x_1^2x_2x_3x_4}$} \ar[ld]^{\pib_3} \ar[rd]^{\pib_1} \ar@(ru,r)[]^{\pib_2=-1} \\
& \txt{$\substack{x_1x_2^2x_4^2+\\x_1x_2^2x_3x_4}$} \ar[rd]_{\pib_2}  \ar@(l,ld)[]_{\txt{$\substack{\pib_1=-1\\ \pib_3=-1}$}} & & x_1^2x_3x_4^2 \ar[rd]_{\pib_1} \ar@(l,ld)[]_{\txt{$\substack{\pib_2=-1\\ \pib_3=-1}$}} & & \txt{$\substack{x_2^2x_3^2x_4+x_1x_2x_3^2x_4\\+x_1x_2^2x_3x_4}$} \ar[ld]^{\pib_3} \ar@(rd,d)[]^{\txt{$\substack{\pib_1=-1\\ \pib_2=0}$}} \\
&& \txt{$\substack{x_1x_3^2x_4^2+x_1x_2x_3x_4^2\\+x_1x_2x_3^2x_4}$} \ar[rd]_{\pib_1} \ar@(ld,d)[]_{\txt{$\substack{\pib_2=-1,\ \pib_3=0}$}} && \txt{$\substack{x_2^2x_3x_4^2+\\x_1x_2x_3x_4^2}$}  \ar[ld]^{\pib_2} \ar@(r,rd)[]^{\txt{$\substack{\pib_1= \pib_3=-1}$}} \\
&&& x_2x_3^2x_4^2 \ar@(r,rd)[]^{\txt{$\substack{\pib_1=\pib_2=-1, \pib_3=0}$}} } \]
\end{example}

\begin{remark}
(i) If $\alpha$ is a composition of $n$ then one has $\H_n(0) \pib_{w_0(\alpha)} x_{\alpha}$ isomorphic to the projective indecomposable $\H_n(0)$-module $\bP_\alpha$. In our earlier work~\cite{H0CF}, we showed that the action of $\H_n(0)$ on the polynomial ring $\FF[x_1,\ldots,x_n]$ induces the same coinvariant algebra as $\SS_n$, and this coinvariant algebra carries the regular representation of $\H_n(0)$ as it decomposes into a direct sum of the projective indecomposable modules $\H_n(0) \pib_{w_0(\alpha)} x_{\alpha}$ for all compositions of $\alpha$.

\noindent(ii)  For any generalized ribbon $\alpha$ of size $n$, one can also embed the $\H_n(0)$-module $\bP_\alpha$ into the \emph{Stanley-Reisner ring of the Boolean algebra of rank $n$}, which admits a natural action of the 0-Hecke algebra $\H_n(0)$ similarly to the polynomial ring $\FF[x_1,\ldots,x_n]$ (see our earlier work~\cite{H0SR}).
\end{remark}

\subsection{Skew quasisymmetric and noncommutative symmetric functions}\label{sec:skew}
Let $A$ be a graded Hopf algebra whose graded components are all finite dimensional. Then there exists a (restricted) dual graded Hopf algebra $A^o$. Every element $f\in A^o$ gives two skewing operators, i.e. for all $a\in A$ whose coproduct is $\Delta(a)=\sum a_1\otimes a_2$, one has
\[
a/f:=\sum a_1f(a_2) \qand f\backslash a:=\sum f(a_1)a_2.
\]
For example, in the self-dual commutative Hopf algebra $\Sym$, the above skew operations both give the skew schur functions $s_{\lambda/\mu}=s_\lambda/s_\mu=s_\mu\backslash s_\lambda$. See Grinberg and Reiner~\cite[\S2.8]{GrinbergReiner}.

One can also apply these skewing operations to $\QSym$ and $\NSym$. Let $\alpha$ and $\beta$ be two compositions. Lam, Lauve, and Sottile~\cite[\S5]{Skew} observed that 
\[ F_{\alpha/\beta}:= F_\alpha / \bs_\beta = 
\begin{cases}
F_{\alpha_{\leq m}}, & {\rm if}\ \alpha_{>m}=\beta, \\
0, & \textrm{otherwise,}
\end{cases} \]
\[ F_{\beta\backslash\alpha}:= \bs_\beta  \backslash F_\alpha = 
\begin{cases}
F_{\alpha_{>m}}, & {\rm if}\ \alpha_{\leq m}=\beta, \\
0, & \textrm{otherwise.}
\end{cases} \]
It was also mentioned in \cite{Skew} that skew ribbon Schur functions do not correspond to skew ribbon shapes in a simple way. Now using our tableau approach to $\NSym$ we obtain an explicit description for skew ribbon Schur functions. 

\begin{proposition}
Let $\alpha$ and $\beta$ be compositions. Then $\bs_{\alpha/\beta}:= \bs_\alpha / F_\beta$ equals $\bs_{\beta\backslash\alpha}:= F_\beta\backslash\bs_\alpha$. Moreover, one has
\[ \bs_{\alpha/\beta} = 
\sum_{\substack{\alpha=\gamma\sqcup\delta \\ \beta\in[\delta]}} \bs_\gamma 
= \sum_{\substack{\alpha=\gamma\sqcup\delta \\ \gamma'\in[\gamma] \\ \beta\in[\delta]}} \bs_{\gamma'} 
\qand \bs_{\beta\backslash\alpha} = 
\sum_{\substack{\alpha=\gamma\sqcup\delta\\ \beta\in[\gamma]}} \bs_\delta
= \sum_{\substack{\alpha=\gamma\sqcup\delta\\ \beta\in[\gamma] \\ \delta'\in[\delta]}} \bs_{\delta'}. \]
\end{proposition}

\begin{proof}
One sees that $\NSym$ is cocommutative, i.e. $\sigma\circ\Delta=\Delta$ where $\sigma$ swaps tensor factors. This implies $\bs_{\alpha/\beta} = \bs_{\beta\backslash\alpha}$ for all compositions $\alpha$ and $\beta$. One obtains the expansion formulas for $\bs_{\alpha/\beta}$ and $\bs_{\beta\backslash\alpha}$ by applying the pairing between $\QSym$ and $\NSym$ to the coproduct formula for $\NSym$ provided in Theorem~\ref{thm:CoprodSchur}.
\end{proof}

The following diagrams give $\bs_{23/2}=\bs_{21}+\bs_{1\oplus2}+\bs_{3}$ and  $\bs_{2\backslash 23}=\bs_{1\oplus2}+\bs_{2\oplus1}$, and thus $\bs_{23/2} = \bs_{2\backslash 23} =\bs_{12}+\bs_{21}+2\bs_3$.
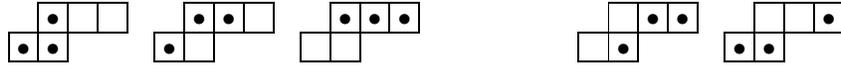
\[ \young(:\bullet\hfill\hfill,\bullet\bullet) \quad \young(:\bullet\bullet\hfill,\bullet\hfill)\quad \young(:\bullet\bullet\bullet,\hfill\hfill)
\qquad\qquad\qquad \young(:\hfill\bullet\bullet,\hfill\bullet)\quad \young(:\hfill\hfill\bullet,\bullet\bullet) \]

\subsection{Combinatorial identities}
Krob and Thibon~\cite{KrobThibon} observed that a cyclic $H_n(0)$-module $M:=\H_n(0) v$ has a length filtration $M=M_0\supseteq M_1\supseteq M_2\supseteq \cdots \supseteq M_k=0$, where $M_i$ is a submodule of $M$ spanned by $\{\pib_w v:\ell(w)\ge i\}$, and this length filtration refines the quasisymmetric characteristic of $M$ to a graded version
\[ \mathrm{Ch}_q(M) := \sum_{0\le i\le k-1} q^i \mathrm{Ch}\left(M_i/M_{i+1}\right). \]
For any generalized ribbon $\alpha$ of size $n$ one can apply this to the $\H_n(0)$-module $\bP_\alpha$ and obtain the following result.

\begin{proposition}\label{prop:ChP}
Let $\alpha=\alpha^1\oplus\cdots\oplus\alpha^k$ be a generalized ribbon of size $n$ with connected components $\alpha^i\models n_i$ for $i=1,\ldots,k$. Let $D_0(\alpha) := D(\alpha^1\rhd\cdots\rhd\alpha^k)$ and $D_1(\alpha) := D(\alpha^1\cdots\alpha^k)$. Then
\[ \sum_{ \substack{w\in \SS_n: \\ D_0(\alpha)\subseteq D(w) \subseteq D_1(\alpha) }} q^{\inv(w)} F_{w^{-1}} = \sum_{z\in \SS^{n_1,\ldots,n_k}} q^{\inv(z)} \prod_{1\le i\le k} \sum_{ \substack{ u_i\in \SS_{n_i}: \\ D(u_i)=D(\alpha^i) }} q^{\inv(u)} F_{u^{-1}}. \]
\end{proposition}

\begin{proof}
The standard tableaux of shape $\alpha$ are in bijection with permutations $w\in\SS_n$ with $D_0(\alpha)\subseteq D(w) \subseteq D_1(\alpha)$, which, by Theorem~\ref{thm:interval}, form an interval under the left weak order of $\SS_n$ starting from the element $w_0(\alpha^1\rhd\cdots\rhd\alpha^k)$. This implies that $\bP_\alpha$ is a cyclic $\H_n(0)$-module generated by the standard tableau of shape $\alpha$ whose reading word is $w_0(\alpha^1\rhd\cdots\rhd\alpha^k)$. Hence
\[ \mathrm{Ch}_q(\bP_\alpha) = \sum_{ \substack{ w\in \SS_n: \\ D_0(\alpha)\subseteq D(w) \subseteq D_1(\alpha) }} q^{\inv(w)-\inv(w_0(\alpha^1\rhd\cdots\rhd\alpha^k))} F_{w^{-1}}. \]

On the other hand, Theorem~\ref{thm:IndP} implies that
\[ \mathrm{Ch}_q(\bP_\alpha) = \sum_{z\in \SS^{n_1,\ldots,n_k}} q^{\inv(z)} \prod_{1\le i\le k} \sum_{ \substack{ u_i\in \SS_{n_i}: \\ D(u_i)=D(\alpha^i) }} q^{\inv(u)-\inv(w_0(\alpha^i))} F_{u^{-1}}. \]
Equating these two expressions for $\mathrm{Ch}_q(\bP_\alpha)$ and observing $w_0(\alpha^1\rhd\cdots\rhd\alpha^k) = w_0(\alpha^1)\cdots w_0(\alpha^k)$ one obtains the desired result.
\end{proof}

Let $\alpha=(\alpha_1,\ldots,\alpha_\ell)\models n$ with partial sums $\sigma_i:=\alpha_1+\cdots+\alpha_i$ for $i=0,1,\ldots,\ell$. The dimension of the $\H_n(0)$-module $\bM_\alpha$ is a multinomial coefficient, which equals the limit as $q\to 1$ of the \emph{q-multinomial coefficient}
\[ \qbin{n}{\alpha}{q} := \frac{[n]!_q}{[\alpha_1]!_q\cdots[\alpha_\ell]!_q} = \sum_{\substack{ w\in\SS_n: \\ D(w)\subseteq D(\alpha) }} q^{\inv(w)} \]
where $[k]_q=1+q+\cdots+q^{k-1}$ and $[k]!_q=[k]_q[k-1]_q\cdots [1]_q$. 
The dimension of the projective indecomposable $\H_n(0)$-module $\bP_\alpha$ is given by the \emph{ribbon number} $r_\alpha$, which is the limit as $q\to1$ of the \emph{$q$-ribbon number}
\[ r_\alpha (q):=\sum_{\substack{ w\in\mathfrak S_n: \\ D(w)=D(\alpha) } } q^{\,{\rm inv}(w)}
=[n]!_q\det\left(\frac1{[\sigma_j-\sigma_{i-1}]!_q}\right)_{i,j=1}^\ell. \]

\begin{corollary}\label{cor:ChP}
Suppose that $\beta$ and $\gamma$ are two compositions of $n$ such that $\beta\cleq \gamma$. Let $(n_1,\ldots,n_k)$ be the composition of $n$ whose descent set equals $D(\gamma)\setminus D(\beta)$, and write $\beta = \alpha^1\rhd\cdots\rhd \alpha^k$ and $\gamma = \alpha^1\cdots\alpha^k$, where $\alpha^i\models n_i$. Then 
\[  \sum_{ \beta\cleq\alpha\cleq\gamma } r_\alpha(q) = \qbin{n}{n_1,\ldots,n_k}{q}r_{\alpha^1}(q)\cdots r_{\alpha^k}(q). \]
\end{corollary}

\begin{proof}
Specialize all fundamental quasisymmetric functions to $1$ in Proposition~\ref{prop:ChP}.
\end{proof}

For example, if $\beta = (2,3,1,2)$ and $\gamma=(2,1,2,1,1,1)$ then 
\[ r_{2312}(q) + r_{21212}(q) + r_{23111}(q) + r_{212111}(q) = \qbin{8}{3,4,1}{q} r_{21}(q) r_{211}(q) r_1(q). \]

One can easily extend the results in this subsection to type B and D.


\end{document}